\documentclass[12pt]{amsart}

\usepackage[all,color]{xy}
\usepackage{pb-diagram}
\usepackage[mathscr]{eucal}
\usepackage{hyperref}
\hypersetup{
    colorlinks=true, %set true if you want colored links
    linktoc=all,     %set to all if you want both sections and subsections linked
    linkcolor=blue,  %choose some color if you want links to stand out
}

\usepackage{xcolor}
\usepackage[normalem]{ulem}

\DeclareMathAlphabet{\mathpzc}{OT1}{pzc}{m}{it}

\usepackage{amsfonts}
\usepackage{amsmath}
\usepackage{amssymb}

\newcommand{\dbar}{\overline{\partial}}

\newcommand{\ddbar}{\sqrt{-1}\partial\dbar}

\newtheorem{theorem}{Theorem}[section]
\newtheorem{proposition}{Proposition}[section]
\newtheorem{lemma}{Lemma}[section]
\newtheorem{conjecture}{Conjecture}[section]

\newtheorem{definition}{Definition}[section]
\newtheorem{corollary}{Corollary}[section]

\newcommand{\QQ}{\mathbb{Q}}

\newcommand{\cM}{\mathcal{M}}
\newcommand{\cS}{\mathcal{S}}

\newcommand{\cX}{\mathcal{X}}

\newcommand{\cZ}{\mathcal{Z}}
\def\KSBA{\mathrm{KSBA}}
\def\WP{\mathrm{WP}}
\def\hilb{\mathrm{Hilb}}

\newcommand{\ti}{\widetilde}

\numberwithin{equation}{section}
\begin{document}

\address{$^*$ Department of Mathematics, Rutgers University, Piscataway, NJ 08854}

\address{$^{**}$ Department of Mathematics and Computer Science, Rutgers University, Newark, NJ 07102}
\address{$^\dagger$Department of Mathematics and Computer Science,  Rutgers University, Newark, NJ 07102}

\thanks{Research supported in
part by National Science Foundation grants DMS-1711439.}

\centerline{ {\bf \normalsize RIEMANNIAN GEOMETRY OF K\"AHLER-EINSTEIN CURRENTS III  } 
% \footnote{Research supported in part by National Science Foundation grant  DMS-1711439 and DMS-1609335 } }
 \footnote{ Research supported in part by National Science Foundation grant  DMS-1711439, DMS-1609335 and  Simons Foundation Mathematics and Physical Sciences-Collaboration Grants, Award Number: 631318. }}

\medskip

\centerline{ \footnotesize COMPACTNESS OF K\"AHLER-EINSTEIN MANIFOLDS OF NEGATIVE SCALAR CURVATURE   }

\bigskip
\bigskip

\centerline{ \small  JIAN SONG$^*$, JACOB STURM$^{**}$, XIAOWEI WANG$^\dagger$}

\bigskip
\bigskip

{\noindent \small A{\scriptsize BSTRACT}. \footnotesize $~$~~Let $\mathcal{K}(n, V)$ be the set of $n$-dimensional compact K\"ahler-Einstein manifolds $(X, g)$ satisfying $\textnormal{Ric}(g)= - g$ with volume bounded above by $V$. We prove that after passing to a subsequence, any sequence $\{ (X_j, g_j)\}_{j=1}^\infty$ in $\mathcal{K}(n, V)$ converges,  in the pointed Gromov-Hausdorff topology, to a finite union of complete K\"ahler-Einstein metric spaces without loss of volume. The convergence is smooth off a closed singular set  of Hausdorff dimension no greater than $2n-4$, and the limiting metric space is biholomorphic to an  $n$-dimensional  semi-log canonical model with its non log terminal locus of complex dimension no greater than $n-1$ removed. We also show that the Weil-Petersson metric extends uniquely to a K\"ahler current with bounded local potentials on the KSBA compactification of the moduli space of canonically polarized manifolds. In particular, the coarse KSBA moduli space has finite volume with respect to the Weil-Petersson metric. Our results are a high dimensional generalization of the well known compactness results for hyperbolic metrics on compact Riemann surfaces of fixed genus greater than one.  }

\bigskip

{\footnotesize  \tableofcontents}

%%%%%%%%%%%%%%%%%%%%%%%%%%%%%%%%%%%%%%%%%%%%%%%%%%%%%%%

\section{Introduction}

The construction  of moduli spaces of Einstein manifolds and their compactifications is a fundamental problem in differential geometry.
For manifolds of real dimension two, the moduli space $\cM_g$ of Einstein metrics (or equivalently the space of complex structures) on a fixed surface $M$ of genus $g\geq 2$,  admits a natural compactification $\overline\cM_g$ known as the Deligne-Mumford compactification \cite{DM}. It is a projective algebraic variety whose
boundary consists of isomorphism classes of stable curves, i.e. curves with finite automorphism group and at worst nodal singularities. Such stable curves admit unique constant curvature metrics with collapsing complete ends near the nodal points. Moreover, a sequence of points $x_j \in
\cM_g$ converges to a point $x_\infty\in \overline\cM_g$ if and only if the corresponding sequence of  Riemann surfaces $X_i$, equipped with their canonical hyperbolic metrics, converges in the pointed Gromov-Hausdorff topology to the nodal surface $X_\infty$.  A main goal in the subject is  extending the theory of  moduli spaces of Riemann surfaces and their natural compactifications to higher dimensional Einstein manifolds $(M, g)$. Here $M$ is a smooth compact manifold and $g$ is a Riemannian metric satisfying  $$\textnormal{Ric}(g) = \lambda g, ~ \lambda= 1, 0, -1.$$

In real dimension $4$, there are many deep results on structures for spaces of Einstein manifolds.  In fact, it is proved in \cite{N, An1, T1, An2} that any sequence of Einstein 4-manifolds $(M_i, g_i)$ with uniform volume lower bound, diameter upper bound and $L^2$-curvature bound will converge to a compact Einstein orbifold, after passing to a subsequence. In the K\"ahler case,  the compactness result holds when $\lambda=0$ 
  if one imposes  uniform  upper and lower volume bounds, and when $\lambda=1$ without any assumption. The $\epsilon$-regularity theorem for 4-manifolds plays a key role in the proof of compactness for such Einstein manifolds.

In higher dimensions, very little is known about the compactness of the moduli space of general Einstein manifolds. But in the K\"ahler setting, recent years have witnessed a surge of activity including several important breakthroughs. The starting point for these developments are the seminal existence results for K\"ahler-Einstein metrics on compact K\"ahler manifolds of negative or vanishing Chern class, established by Yau in his solution to the Calabi conjecture (also see \cite{A}) and the uniqueness results due to Calabi.  When the first Chern class of $M$ is positive, uniqueness of K\"ahler-Einstein (modulo automorphisms) was established by Bando-Mabuchi, but  obstructions to the existence of  
such metrics have been known for many years. The fundamental conjecture of Yau-Tian-Donaldson, which predicts that the existence of a K\"ahler-Einstein metric on $M$ is equivalent to the algebraic condition of its $K$-stability, was recently solved \cite{CDS1, CDS2, CDS3, T4} and gives a complete answer to the existence problem on Fano manifolds. The partial $C^0$-estimate, proposed by Tian and proved for surfaces in \cite{T1},  is a fundamental tool for studying analytic and geometric degenerations of K\"ahler-Einstein metrics in the non-collapsing case, and  was established for higher dimensions by Donaldson-Sun \cite{DS1}.  As a consequence, any sequence of $n$-dimensional K\"ahler-Einstein manifolds $(M_i, g_i)$ with $\textnormal{Ric}(g_i) =  g_i$ must converge in Gromov-Hausdorff topology, after passing to a subsequence, to a singular  K\"ahler-Einstein metric space homeomorphic to a $\mathbb{Q}$-Fano variety (\cite{DS1}).   In this case, the global uniform non-collapsing for  all $(M_i, g_i)$ immediately follows from the volume comparison theorem and the fact that the first Chern is an integral class. The partial $C^0$-estimate can also be viewed as effective base point free theorem in the Riemannian setting and can be applied to understand the algebraic structures of the limiting K\"ahler-Einstein metric spaces by developing a connection between the analytic and geometric K\"ahler-Einstein metrics and the algebraic Bergman metrics.

However, when $\lambda=-1$, as one sees already in the case of Riemann surfaces of higher genus,   Einstein metrics   collapse at the complete ends and the limiting spaces are no  longer compact in general. Such complete ends correspond to the nodes arising from algebraic degenerations of smooth curves with genus greater than one. The major result in real dimension $4$ for compactness of K\"ahler-Einstien manifolds of negative scalar curvature is due to Cheeger-Tian \cite{CT},  which was established much later than the corresponding compactness result in the positive scalar curvature setting \cite{An1, N, T1}. A key tool in the collapsing setting is the chopping technique  of \cite{CG} and a refined $\epsilon$-regularity theorem for $4$-folds to establish a non-collapsing result for Einstein $4$-manifolds with uniformly bounded $L^2$-curvature. 

In the complex case, when $(M_i, J_i, g_i)$ is a sequence of K\"ahler-Einstein surfaces  on a smooth manifold $M$ satisfying $\textnormal{Ric}(g_i) = - g_i$ with a uniform volume upper bound,  it is proved in \cite{CT} that $(M_i, J_i, g_i)$ converges in the pointed Gromov-Hausdorff topology, after passing to a subsequence, to a finite disjoint union of complete orbifold K\"ahler-Einstein  surfaces without loss of total volume.  It is proved in \cite{M, MZ, Ye1} that any complete K\"ahler manifold of finite volume and negative Ricci curvature must be quasi-projective if the sectional curvatures are bounded. Its orbifold analogue immediately implies that each component in the limiting metric space of K\"ahler-Einstein surfaces of negative scalar curvature must be quasi-projective and in fact, admits a unique projective compactifcation which is a log canonical pair \cite{Ye2}.

The main purpose of this paper is to establish the analogue, for K\"ahler manifolds of higher dimension,  of the classical compactness results for constant curvature metrics on high genus curves as well as the result of Cheeger-Tian \cite{CT} on K\"ahler-Einstein surfaces of negative scalar curvature. In particular, we  show the analytic compactification agrees precisely with  $\overline{\cM}_{\rm KSBA}$, the algebraic compactification of the moduli space of K\"ahler manifolds with negative first Chern class \cite{KSB,Al}.

Special degenerations of  K\"ahler-Einstein manifolds of negative first Chern class were first studied in \cite{T1,LL, Ru1} when the special fibres have only simple normal crossing singularities inside a smooth total space, and smooth pointed Cheeger-Gromov convergence was shown to hold away from such simple normal crossing singularities. The first general compactness result is established by the first author in \cite{S4} for higher dimensional K\"ahler-Einstein manifolds of negative scalar curvature. More precisely,  it is proved in \cite{S4} that  if $f: \mathcal{X} \rightarrow B^*$  is a projective family of K\"ahler manifolds of negative first Chern class over a punctured disc $B^*=\{ t\in \mathbb{C}~|~0<|t|<1 \}$,  the fibres $\mathcal{X}_t=f^{-1}(t)$ equipped with the unique K\"ahler-Einstien metrics $g_t$ must converge in pointed Gromov-Hausdorff topology to a quasiprojective K\"ahler-Einstein variety as $t\rightarrow 0$ and the projective compatification of such a limiting variety coincides with the unique semi-log canonical model as the central fibre for the algebraic log canonical closure of $f: \mathcal{X}\rightarrow B^*$. In particular, any boundary point of the algebraic compatification of the moduli space of smooth canonical models corresponds to a unique Riemannian complete K\"ahler-Einstein metric space.    In this paper, we aim to generalize the result of \cite{S4} to algebraic degeneration of K\"ahler-Einstein manifolds of negative scalar curvature over higher dimensional base, which will allow us to to solve the general compactness problem. In addition, we establish new $C^0$ estimates for the Weil-Petersson potential on the compactified KSBA moduli space.

We start by introducing the main objects of interests, the space of K\"ahler-Einstein spaces with negative first Chern class with fixed dimension and upper volume bound.
 \begin{definition}\label{kespace}
Let  $\mathcal{K}(n, V)$ be the space of K\"ahler-Einstein manifolds  defined by
$$\mathcal{K}(n, V)= \{ (X, J, g)~|~(X, J, g)~ \textnormal{is ~ K\"ahler}, ~\dim_{\mathbb{C}} X =n, ~ \textnormal{Ric}(g) = - g, ~ \textnormal{Vol}(X, g) \leq V\}.$$
\end{definition}

We now state our first main result which concerns the compactness of the space $\mathcal{K}(n, V)$.
\begin{theorem} \label{main1}  For any sequence $\{(X_j, J_j, g_j)\}_{j=1}^\infty \subset \mathcal{K}(n, V)$, after passing to a subsequence, there exist $m\in \mathbb{Z}^+$ and a sequence of $P_j = (p_{1, j}, p_{2,j}, ..., p_{m, j}) \in \amalg_{k=1}^m   X_{{j}}$, such that $(X_j, J_j, g_j, P_j)$ converge in the pointed Gromov-Hausdorff topology to a finite disjoint union of metric spaces
\begin{equation}
(\mathcal{Y}, d_{\mathcal{Y}})  = \amalg_{k=1}^m (\mathcal{Y}_k, d_k)
\end{equation}
satisfying the following.

\begin{enumerate}

\item For each $k$, $\left\{ (X_j, J_j, g_j, p_{k, j}) \right\}_{j=1}^\infty$ converges smoothly to an $n$-dimensional K\"ahler-Einstein manifold $(\mathcal{Y}_k\setminus \mathcal{S}_k, \mathcal{J}_k, \mathpzc{g}_k)$ away from the singular set $\mathcal{S}_k$ of $(\mathcal{Y}_k, d_k)$.  In particular, $\mathcal{S}_k$ is closed and has Hausdorff dimension no greater than $2n-4$.

\medskip

\item The complex structure $\mathcal{J}_k$ on $\mathcal{Y}_k \setminus \mathcal{S}_k $ uniquely extends to $\mathcal{Y}_k$ and $(\mathcal{Y}_k, \mathcal{J}_k)$ is an $n$-dimensional  normal quasi-projective variety with at worst log terminal singularities.

\medskip

\item  $(\mathcal{Y}_k, d_k)$ is the metric completion of $(\mathcal{Y}_k\setminus \mathcal{S}_k, \mathpzc{g}_k)$ and $\mathpzc{g}_k$ uniquely extends to a K\"ahler current on $\mathcal{Y}_k$ with locally bounded K\"ahler potentials.

\medskip

\item $\sum_{k=1}^m \textnormal{Vol}(\mathcal{Y}_k, d_k)= \lim_{j\rightarrow \infty} \textnormal{Vol}(X_j, g_j)$.

\medskip

\item There exists a unique projective compactification $\overline{\mathcal{Y}}$ of $\mathcal{Y}$ such that $\overline{\mathcal{Y}}$ is a semi-log canonical model and $\overline{\mathcal{Y}}\setminus \mathcal{Y}$ is the non log terminal locus of $\overline{\mathcal{Y}}$.

\end{enumerate}

\end{theorem}

The  number $m$ of components in the limiting space is bounded above by a constant $M=M(n, V)$ and $\lim_{V\rightarrow \infty} M=  \infty$.  Moreover,  it follows from  \cite[Theorem 1.9]{HMX2013} that the volume of each component $(\mathcal{Y}_k, d_k)$ in the limiting metric space is bounded below by a constant $\nu=\nu(n)>0$ that only depends on the dimension $n$. 
A smooth canonical model is a complex projective manifold whose canonical bundle is ample, or equivalently, with negative first Chern class. There is always a unique K\"ahler-Einstein metric in the canonical class of a smooth canonical model. Semi-log canonical models (see Definition \ref{semilog}) are  high-dimensional analogues of high genus nodal curves and are  generalizations of canonical models with hypersurface singularities. Theorem \ref{main1} is a also generalization of the holomorphic compactness result  in \cite{S4}.

Let us say a few words about the proof. As in \cite{S4}, we first need to control the growth rate towards $-\infty$ of the K\"ahler-Einstein potential near the log canonical locus. To overcome the difficulties posed by a higher dimensional base, we use semi-stable reduction together with an analytic form of Inverse of Adjunction to obtain the necessary control. Then, following the road map in \cite{S4}, we prove a uniform non-collapsing condition for $(X_j, g_j)\in \mathcal{K}(n, V)$, more precisely, we will show that there exist $c=c(n, V)>0$ and $p_j \in X_j$   such that for all $j\geq 1$
\begin{equation}\label{uninoncol}
{\rm Vol}_{g_j} (B_{g_j}(p_j, 1) )>c .
\end{equation}
The Cheeger-Colding theory cannot be applied without  condition (\ref{uninoncol}). Such a non-collapsing condition is achieved by uniform analytic estimates using the algebraic semi-stable reduction.  Instead of considering a general sequence of K\"ahler-Einstein manifolds with fixed dimension and volume upper bound, we can consider an algebraic family of canonical models with fixed Hilbert polynomial, and study the Riemannian geometric behavior of K\"ahler-Einstein metrics on algebraic degeneration of such manifolds towards semi-log canonical models.   Therefore it suffices to prove the compactness for K\"ahler-Einstein manifolds as a holomorphic family over a projective variety $\mathcal{S}$. One of the main contributions in this paper is to derive a local analytic estimate for $\dim \mathcal{S} \geq 2$, improving the results in \cite{S4} for $\dim \mathcal{S} =1$.  Roughly speaking, we aim to turn the bigness of canonical class into non-collapsing of the Riemannian metrics.  With the noncollapsing condition (\ref{uninoncol}), the partial $C^0$-estimates will naturally hold by the fundamental work in \cite{T1, DS1}. One of the difficult issues here for degeneration of canonical models is that, the diameter will in general tend to infinity and there must be collapsing at the complete ends in the limiting metric space. We will have to estimate the distance near the singular locus of special fibres of the degeneration. It turns out by applying the analytic and geometric techniques developed in \cite{S4} that the general principle for geometric K\"ahler currents applies, i.e., boundedness of the local potential is equivalent to boundedness of distance to a fixed regular base point. Such a principle is achieved by building a local Schwarz lemma with suitable auxiliary K\"ahler-Einstein currents and the $L^\infty$-estimates for degenerate complex Monge-Amp\`ere equations from the capacity theory \cite{Kol1, EGZ}. Our proof also relies on the the recent progress in birational geometry which establishes the KSBA compatification of the moduli space of canonical models.  Theorem \ref{main1}  gives the expected relation between differential geometric moduli spaces and algebraic moduli spaces of smooth canonical models. 

There are still many open questions (c.f. \cite{HMX, HX, K1, K2}). The standard partial $C^0$-estimates fail near the log canonical and semi-log canonical locus of the semi-log canonical models because of collapsing, and the Riemannian geometric limit will push such singularities to infinity which limits our understanding of the metric and the geometric structure. Naturally, one hopes to use geometric $L^2$-theory to compactify and glue the complete metric spaces in the limit along the semi-log canonical locus to recover the algebraic semi-log canonical models.  Also it seems to be more natural to compactify the moduli space by using the compactness of K\"ahler-Einstein metrics without running the relative minimal model program (MMP). It is possible that  most results in Theorem \ref{main1} can be proved without using MMP, although it is not clear how to identify the limiting metric space as an algebraic variety since the partial $C^0$-estimates as well as geometric convergence theory are not quite available in the collapsing case.

Yau has proposed studying the metric completion of the Weil-Petersson metric on the moduli spaces as an approach for compactification. We  apply Theorem \ref{main1} and its proof to obtain some analytic understanding the relative canonical sheaf as well as the Weil-Petersson  currents. To explain our next result we first make the following definition.

\begin{definition} \label{stabfib} A  flat projective morphism $\pi: \mathcal{X} \rightarrow \mathcal{S}$ between normal varieties $\mathcal{X}$ and $\mathcal{S}$ is said to be a stable family of  canonical models, i.e., K\"ahler manifolds of negative first Chern class,  if the following hold.

\begin{enumerate}

\item The relative canonical sheaf $K_{\mathcal{X}/\mathcal{S}}$ is $\mathbb{Q}$-Cartier and  is $\pi$-ample.

\medskip

\item The general fibres are smooth canonical models and the special fibres are semi-log canonical models.

\end{enumerate}
We let $\mathcal{S}^\circ$ be  the set of smooth points of $\mathcal{S}$ over which $\pi$ is smooth and $\mathcal{X}^\circ = \pi^{-1}(\mathcal{S}^\circ).$

\end{definition}

In this paper, we always assume the general fibres of a stable family are smooth, i.e., the family $\pi|_{\cX^\circ}: \cX^\circ \rightarrow \cS^\circ$ in Definition \ref{stabfib} is a holomorphic family of canonically polarized projective manifolds. For each $t\in \cS^\circ$, there exists a unique K\"ahler-Einstein form $\omega_t\in c_1(\cX_t)$ on the fibre $\cX_t = \pi^{-1}(t)$ and one can define the hermitian metric on the relative canonical bundle $K_{\cX^\circ/\cS^\circ}$ by $h_t = (\omega_t^n)^{-1}.$ It is proved by Schumacher \cite{Sch} and Tsuji \cite{Ts} using different methods that
$$\textnormal{Ric}(h) = - \ddbar \log h$$
is nonnegative on $\cX^\circ$ and  $\textnormal{Ric}(h)$ is strictly positive if the family is nowhere infinitesimally trivial.
It is further shown in \cite{Sch} that $h$ can be uniquely extended to a non-negatively curved singular hermitian metric on $K_{\cX/\cS}$ with analytic singularities, which implies that the relative canonical bundle $K_{\cX/\cS}$ is pseudo-effective.  Using the proof of Theorem \ref{main1}, we can give a sharp description of  the analytic singularities of $h$. In fact, $h$ has vanishing Lelong number everywhere on $\cX$ and it tends $-\infty$ exactly at the semi-log and log canonical locus of the special fibres as semi-log canonical canonical models of $\pi: \cX \rightarrow \cS$. In particular, this gives an analytic proof of Fujino's theorem 
\cite{Fujino}
which says that the relative canonical bundle $K_{\cX/\cS}$ is nef.

\begin{theorem} \label{main2} Let $\pi: \mathcal{X} \rightarrow \mathcal{S}$ be a stable family of $n$-dimensional   canonical models over a projective normal variety $\mathcal{S}$.  Let $\omega_t$ be the unique K\"ahler-Einstein metric on $\mathcal{X}_t$ for $t\in \mathcal{S}^\circ$ and $h$ be the hermitian metric on the relative canonical sheaf $K_{\mathcal{X}^\circ/\mathcal{S}^\circ}$ defined by
$$h_t =( \omega_t^n)^{-1}.$$
The curvature $\theta= \ddbar \log h$ of $(K_{\mathcal{X}^\circ/\mathcal{S}^\circ}, h)$  extends uniquely to a closed nonnegative $(1,1)$-current on $\mathcal{X}$ with vanishing Lelong number everywhere. Furthermore, 

\smallskip

\begin{enumerate}
\item $\theta|_{\mathcal{X}_t}=\omega_t$ for $t\in \mathcal{S}^\circ$.

\medskip

\item
$\theta|_{\mathcal{X}_t}$ is the unique canonical K\"ahler-Einstein current on the semi-log canonical model $\mathcal{X}_t$ for any $t\in \mathcal{S}\setminus \mathcal{S}^\circ$.

\medskip

\item  $K_{\mathcal{X}/\mathcal{S}}$ is nef on $\mathcal{X}$.

\end{enumerate}
Moreover, if all the fibres of $\pi: \mathcal{X} \rightarrow \cS$ have at worst log terminal singularities, then $\theta$    has bounded local potentials on $\mathcal{X}$.

\end{theorem}

The Weil-Petersson metric on the moduli space of K\"ahler-Einstein manifolds with negative first Chern class is the $L^2$-metric of the harmonic $(0,1)$-form with coefficients in the tangent bundle and it is a natural generalization of the Weil-Petersson metric for the Teichm\"uller space. It is shown in \cite{Ko} that the Weil-Petersson metric is  K\"ahler over the base of families of smooth canonical models. The following theorem shows that the Weil-Petersson metric can be uniquely extended globally to the base of any stable family. 
\begin{theorem} \label{main3} Let $\pi: \mathcal{X} \rightarrow \mathcal{S}$ be a stable family of $n$-dimensional   canonical models over an $m$-dimensional projective normal variety $\mathcal{S}$. 
 Then the Weil-Petersson metric $\omega_{WP}$ on $\mathcal{S}^\circ$  extends uniquely to a non-negative closed $(1,1)$-current on $\mathcal{S}$ with bounded local potentials, i.e., for any point $p \in \mathcal{S}$, there exists an open neighborhood $U$ of $p$ in $\mathcal{S}$ and a bounded plurisubharmonic function $\psi$ in $U$ such that
$$
 \omega_{WP}=\ddbar \psi.
$$
 In particular, the volume of $(\mathcal{S}, \omega_{WP})$ is finite and 
 $$\int_\mathcal{S} (\omega_{WP})^m=\int_{\mathcal{S}^\circ} (\omega_{WP})^m \in \mathbb{Q}^+.$$

\end{theorem}

In fact, the extended Weil-Petersson metric on $\mathcal{S}$ in Theorem \ref{main3} is the curvature of the CM line bundle on $\cS$ defined by the Deligne Pairing   $\langle K_{\mathcal{X}/\mathcal{S}},K_{\mathcal{X}/\mathcal{S}}...,K_{\mathcal{X}/\mathcal{S}}\rangle$ of $K_{\mathcal{X}/\mathcal{S}}$ with itself $n+1$ times. 

Let $(X, J)$ be an $n$-dimensional  smooth canonical model of general type with complex structure $J$. Let $\overline{\cM}_\KSBA$ be the KSBA compactification of the moduli space of complex structures on $X$. Then
$\overline{\cM}_\KSBA$
is a projective variety and the CM line bundle $\mathcal{L}$ induced by the Deligne pairing is well-defined on $\overline{\cM}_\KSBA$ (c.f. \cite{PX}). The Weil-Petersson metric $\omega_{WP}$ is also a well-defined smooth K\"ahler metric on $\left({\cM}_\KSBA\right)^\circ$,  a Zariski open set of ${\cM}_\KSBA$ (see section 9). In particular, $\omega_{WP}$  is the curvature of $\mathcal{L}$ on $\left({\cM}_\KSBA\right)^\circ$. The following corollary gives the extension of $\omega_{WP}$ on $\overline{\cM}_\KSBA$.

\begin{corollary} \label{main4} The Weil-Petersson $\omega_{WP}$ extends uniquely to a nonnegative closed $(1,1)$-current on the coarse moduli space $\overline{\cM}_\KSBA$ with bounded local potentials and 
$$\omega_{WP}\in c_1(\mathcal{L}_{\rm CM}),$$ 
where $\mathcal{L}_{\rm CM}$ is the CM bundle on $\overline{\cM}_\KSBA$. In particular, the volume of $\overline{\cM}_\KSBA$ with respect to the Weil-Petersson metric is a rational number, i.e., 
\begin{equation}
\int_{\overline{\cM}_\KSBA} \left( \omega_{WP}\right)^{d} = \int_{\left(\overline{\cM}_\KSBA\right)^\circ} \left( \omega_{WP}\right)^{d} = [c_1(\mathcal{L})]^{d}\in \mathbb{Q}^+,
\end{equation}
where $d= \dim \overline{\cM}_\KSBA$.

\end{corollary}

In general, $\overline{\cM}_\KSBA$ is not a normal variety, but the integral $\int_{\overline{\cM}_\KSBA} \left( \omega_{WP}\right)^{d}$ can be  computed on the normalization of $\overline{\cM}_\KSBA$ and the volume is independent of the choice of normalization. We will explain how the Weil-Petersson metric is defined on the moduli space.   Corollary \ref{main4} also implies that $\mathcal{L}_{\rm CM}$ is big and nef since $\omega_{WP}\geq 0$ with bounded local potentials. In fact, it is proved \cite{PX} that $\mathcal{L}_{\rm CM}$ is ample.  

It is expected that the Weil-Petersson metric has finite distance for the moduli space of canonical models as in the case of high genus curves. This is confirmed in \cite{T1, Ru1} for special degeneration of smooth canonical models whose central fibre has only normal crossing singularities. Naturally, one also expects the local potentials of the Weil-Petersson metric to be always bounded.  Therefore we propose the following conjecture which is an extension of Theorem \ref{main3}.

\begin{conjecture} \label{con1} The Weil-Petersson metric $\omega_{WP}$ on $\mathcal{S}^\circ$ extends uniquely  to a non-negative closed $(1,1)$-current on $\mathcal{S}$ with continuous local potentials on $\mathcal{S}$ and the  completion of $(\mathcal{S}^\circ, \omega_{WP})$ is homeomorphic to $\mathcal{S}$ if $\pi: \mathcal{X}\rightarrow \mathcal{S}$ is nowhere infinitesimally trivial.

\end{conjecture}
 We are able to partially verify Conjecture \ref{con1} and prove the continuity of the Weil-Petersson potentials  in \cite{SSW}.

Remark: We can also let $S=\overline{\cM}_\KSBA$ be the moduli space in Conjecture \ref{con1}. 

There have been many results in the study of canonical K\"ahler metrics on projective varieties of non-negative Kodaira dimension using the K\"ahler-Ricci flow or deformation of K\"ahler metrics of Einstein type \cite{ST1, ST2, S2, FGS}. Such deformations, with both analytic and geometric estimates,  have applications in algebraic geometry, for example, an analytic proof of Kawamata's base point free theorem for minimal models of general type is obtained in \cite{S3}. More generally, the K\"ahler-Ricci flow is closely related to the minimal model program and the formation of finite time singularities is an analytic geometric transition of solitons for birational flips. A detailed program is laid out in \cite{ST3} with partial geometric results obtained in \cite{SW1, SW2, SY, S1}. Combining Theorem \ref{main1} and the results in \cite{S4}, there always exists a unique geometric K\"ahler-Einstein metric on canonical models which either have a smooth minimal model or are smoothable. For general semi-log canonical models, we expect further applications towards understanding algebraic and geometric structures of the singularities which arise, but significant work remains.

 We give a brief outline of the paper. In \S2 and \S3, we review the definition of canonical K\"ahler-Einstein currents on semi-log canonical models and state  the semi-stable reduction for stable families of canonical models.  We prove the local $C^0$-estimates in \S3. In \S5 and \S6, we adapt the proof in \cite{S4} to obtain necessary analytic and geometric estimates.  Theorem \ref{main1} is proved in \S7. Finally, we prove Theorem \ref{main2} in \S8 and Theorem \ref{main3} in \S9.

 %%%%%%%%%%%%%%%%%%%%%%%%%%%%%%%%%%%%%%%%%%%%%%%%%%%%%%%

\section{K\"ahler-Einstein currents on semi-log canonical models}

In this section, we  review the definitons semi-log canonical models and the algebraic degeneration of canonical models of general type. First, let us recall the definition for log canonical singularities of a projective variety.

\begin{definition} \label{sing} Let $X$ be a normal projective variety such that $K_X$ is a $\mathbb{Q}$-Cartier divisor. Let $ \pi :  Y \rightarrow X$ be a log resolution and $\{E_i\}_{i=1}^p$ the irreducible components of the exceptional locus $Exc(\pi)$ of $\pi$. There there exists a unique collection $a_i\in \mathbb{Q}$ such that
$$K_Y = \pi^* K_X + \sum_{i=1}^{ p } a_i E_i .$$ Then $X$ is said to have
\begin{enumerate}

\item[$\bullet$] terminal singularities if  $a_i >0$, for all $i$.

\medskip

\item[$\bullet$] canonical singularities if $a_i \geq 0$, for all $i $.

\medskip

\item[$\bullet$]  log terminal singularities if $a_i > -1$, for all $i $.

\medskip

\item[$\bullet$]  log canonical singularities if $a_i \geq -1$, for all $i$.

\medskip

\end{enumerate}
A projective normal variety $X$ is said to be a canonical model if $X$ has canonical singularities and $K_X$ is ample.

\end{definition}

Smooth canonical models are simply K\"ahler manifolds of negative first Chern class. There always exists a unique K\"ahler-Einstein metric on smooth canonical models \cite{A, Y1}.  The notion of semi-log canonical models is crucial  in the study of degenerations of smooth canonical models (c.f. \cite{K1, Kov}).

\begin{definition} \label{semilog} A reduced projective variety $X$ is said to be a semi-log canonical model  if

\begin{enumerate}
\item $K_X$ is an ample $\mathbb{Q}$-Cartier divisor,
\medskip

\item $X$ has only ordinary nodes in codimension $1$,

\medskip

\item For any log resolution $\pi: Y \rightarrow X$,
$$K_Y = \pi^* K_X  + \sum_{i=1}^I a_i E_i - \sum_{j=1}^J  F_j,$$
where  $E_i$ and $F_j$, the irreducible components of exceptional divisors,  are smooth divisors of  normal crossings with $a_i > -1$.
\end{enumerate}
We denote the algebraic closed set  $\textnormal{LCS}(X) =Supp\left(  \pi\left(\bigcup_{j=1}^J F_j \right) \right)$ in $X$ to be the locus of non-log terminal singularities of $X$. We also let $\mathcal{R}_X$ be the nonsingular part of $X$ and $\mathcal{S}_X$ the singular set of $X$.

\end{definition}
  Implicit in our assumptions is the requirement that $X$ is $\mathbb{Q}$-Gorenstein and satisfies Serre's $S_2$ condition] (see \cite{Kov} for more details on such algebraic notions).  The non-log terminal locus is the set of singular points of $X$ with their discrepancy equal to $-1$. Analytically, this locus is the set of points near which the adapted volume measure fails to be locally integrable.

In general, $X$ is not normal. Let $\nu: X^\nu \rightarrow X$ be the normalization of $X$ and
$$K_{X^\nu} = \nu^* K_X - {\rm cond}(\nu), $$
where ${\rm cond}(\nu)$ is a reduced effective divisor known as the the conductor. In fact, ${\rm cond}(\nu)$ is the inverse image of codimension-one ordinary nodes. Now $K_{X^\nu}+ {\rm cond}(\nu)$ is a big and semi-ample divisor on $X^\nu$ and so the pair $(X^\nu, K_{X^\nu}+{\rm cond}(\nu))$ has log canonical singularities.
Let $\pi^\nu: Y \rightarrow X^\nu$ be a log resolution. Then $\pi = \nu \circ \pi^\nu: Y \rightarrow X$ is also a log resolution and
$$K_Y = \pi^*K_X + \sum a_i E_i  - \sum F_j, ~ a_i >-1, $$
where $E_i$, $F_j$ are the exceptional prime divisors of $\pi$.  In other words, the normalization $(X^\nu)$ of $X$ gives  rise to a log canonical pair $(X^\nu, {\rm cond}(\nu) )$. In particular,  if one aims to construct a K\"ahler-Einstien metric $g_{KE}$ on $X$, $g_{KE}$ can be pulled backed to a suitable K\"ahler current on $X^\nu$ in the class of $K_{X^\nu} + {\rm cond}(\nu)$.

\begin{theorem}\label{song} \cite{S4} Let $X$ be a semi-log canonical model with $\dim_{\mathbb{C}} X = n$. There exists a unique K\"ahler current $\omega_{KE} \in -c_1(X)$ such that

\begin{enumerate}

\item $\omega_{KE}$ is smooth on $\mathcal{R}_X$, the nonsingular part of $X$,  and it satisfies the K\"ahler-Einstein equation on $\mathcal{R}_X$
$$\textnormal{Ric}(\omega_{KE} ) = - \omega_{KE}.$$
\item $\omega_{KE}$ has bounded local potentials on the quasi-projective variety $X\setminus \textnormal{LCS}(X)$, where $\textnormal{LCS}(X)$ is the non-log terminal locus of $X$. More precisely, let $$\Phi: X \rightarrow \mathbb{CP}^N$$ be a projective embedding of $X$ by a pluricanonical system $H^0(X, mK_X)$ for some $m\in \mathbb{Z}^+$ and let $\chi = \frac{1}{m} \omega_{FS}|_{X}$, where $\omega_{FS}$ is the Fubini-Study metric on $\mathbb{CP}^N$.  Then  $\omega_{KE} = \chi + \ddbar \varphi_{KE}$ for some $\varphi _{KE}\in \textnormal{PSH}(X, \chi)$ such that $$\varphi_{KE} \in L^\infty_{loc} (X\setminus \textnormal{LCS}(X)), ~~ \varphi_{KE} \rightarrow - \infty ~\textnormal{near}~ \textnormal{LCS}(X) . $$

\item Let $D$ be any effective divisor of $X$ such that the support of $D$ contains the singularities of $X$. Then for any $\epsilon>0$, there exists $C_\epsilon>0$ such that on $X$,
$$\varphi_{KE} \geq \epsilon \log |\sigma_D|^2_{h_D} - C_\epsilon, $$
where $\sigma_D$ is a defining section of $D$ and $h_D$ is a fixed smooth hermitian metric on the line bundle associated to $D$.

\medskip

\item The Monge-Amp\`ere mass $\omega_{KE}^n$ does not charge mass on the singularities of $X$ and $$\int_{X} \omega_{KE}^n = [K_X]^n. $$

\end{enumerate}

\end{theorem}

The K\"ahler-Einstein currents on semi-log canonical models  were first constructed in \cite{B} using a variational method and they coincide with the K\"ahler-Einstein currents constructed in Theorem \ref{song} by the uniqueness. Also the quasi-plurisubharmonic function $\varphi_{KE}$ in Theorem \ref{song} has vanishing Lelong number everywhere on $X$ because  $\varphi_{KE}$ is milder than any log poles from (3) in Theorem \ref{song}. The estimates on local boundedness of the K\"ahler-Einstein potentials in \cite{B} are not sufficient to study the Riemannian geometric degeneration  for a stable degeneration of smooth canonical models. Our approach is to combine the fundamental results \cite{Kol1, EGZ} and the maximum principle with suitable barrier functions. This helps us to obtain local $L^\infty$-estimates of local potentials away from the non-log terminal locus of $X$.

By the uniqueness of the K\"ahler-Einstein current in Theorem \ref{song}, we can incorporate Theorem \ref{song} into  the definition of canonical K\"ahler-Einstein currents on  semi-log canonical models.

\begin{definition} \label{ckem} Let $X$ be a semi-log canonical model with $\dim_{\mathbb{C}} X = n$. We define the canonical K\"ahler-Einstein current on $X$ to be the current $\omega_{KE}$ in Theorem \ref{song}.

\end{definition}
%

%%%%%%%%%%%%%%%%%%%%%%%%%%%%%%%%%%%%%%%%%%%%%%%%%%%%%%%%%%%%

%%%%%%%%%%%%%%%%%%%%%%%%%%%%%%%%%%%%%%%%%%%%%%%%%%%%%%%

\section{Stable families of canonical models and semi-stable reduction}

In this section, we will give some algebraic background of KSBA families and semi-stable reduction. Let
$$\pi: \cX \rightarrow \cS $$
be a stable family of smooth canonical models as  in Definition \ref{stabfib}. Without loss of generality, we can assume $\mathcal{S}$ is smooth by resolving singularities of $\mathcal{S}$  because the new family after resolution of singularities on $\mathcal{S}$ is again a stable family of smooth canonical models with the same fibres.

The following lemma is well-known and follows immediately from the definition of a flat morphism.
\begin{lemma} If  $p$ is a smooth point of a fibre $\mathcal{X}_t = \pi^{-1}(t)$ for some $t\in \mathcal{S}$, $s$ is also a smooth point of the total space $\mathcal{X}$.

\end{lemma}

The following semi-stable reduction is due to Ambromovich-Karu \cite{AK} and Adiprasito-Liu-Temkin \cite{ALT}.

\begin{theorem} \label{semired}  Let $\pi: \mathcal{X} \rightarrow \mathcal{S}$ be a surjective flat morphism of projective varieties with   generically integral fibres. Then $\mathcal{X} \rightarrow \mathcal{S}$ admits a semistable reduction. That is, there exist a proper surjective generically finite morphism  $f: \mathcal{S}' \rightarrow \mathcal{S} $ and a proper birational morphism $\Psi: \mathcal{X}' \rightarrow \mathcal{X} \times_\mathcal{S} \mathcal{S}'$ for some smooth $\cX'$ as in the following diagram
\begin{equation}
\begin{diagram}
\node{\mathcal{X}'} \arrow{se,l}{ \pi' }  \arrow{e,t}{\Psi}   \node{\cZ:=\mathcal{X}\times_\mathcal{S} \mathcal{S}'}  \arrow{e,t}{f'} \arrow{s,r}{\pi_{\mathcal{Z}}}   \node{\mathcal{X} } \arrow{s,r}{\pi} \\
\node{}      \node{\mathcal{S}'} \arrow{e,t}{f}  \node{\mathcal{S}}
\end{diagram}
\end{equation}
%.
satisfying the following.

For any  $x' \in \mathcal{X}'$ and $s' = f'(x') \in \mathcal{S}'$, there exist formal holomorphic coordinates $(x_1, ..., x_{n+d})$, $(t_1, ..., t_d)$   centered at $x'$ and $s'$ respectively, such that $f'$ is given by near $x'$
\begin{equation} \label{coordinates}
t_i = \prod_{j=m_{i-1}+1}^{m_i} x_j,
\end{equation}
for some $0=m_0< m_1<...<m_d   \leq   n+d$.

\end{theorem}

The coordinates $t_1, ..., t_d$, $x_1, ..., x_{n+d}$ are obtained from global coordinates from toroidal embeddings and the hyperplanes defined by $x_i=0$ in one local open neighborhood must also appear as zeros of such local toric coordinates in every local neighborhood. Weak semi-stable reduction is established by Karu-Ambramovich in \cite{AK}, and in their result $\tilde\cX$ is not necessarily smooth. It is conjectured in \cite{AK} that $\tilde{X}$ can be made smooth. The conjecture was proved for $n=3$ in \cite{Ka}  and was recently  settled   completely   in \cite{ALT}.  Weak semi-stable reduction suffices for our purpose, but we will use the semi-stable reduction result of \cite{ALT}  which slightly simplifies the argument .  After the base alteration $\mathcal{S}'\rightarrow \mathcal{S}$, 
$$\pi_{\mathcal{Z}}: \mathcal{Z} \rightarrow  \mathcal{S}' $$ is still a stable family of smooth canonical models.
We will now focus on the following diagram
\begin{equation}\label{diag3}
\begin{diagram}
\node{\mathcal{X}'} \arrow{se,l}{ \pi' }  \arrow{e,t}{\Psi}   \node{\cZ}   \arrow{s,r}{\pi_{\mathcal{Z}}}    \\
\node{}      \node{\cS'=B}
\end{diagram}
\end{equation}
where $B$ is the unit ball in $\mathbb{C}^d$, because we can always resolve singularities of $\mathcal{S}'$ and assume $\mathcal{S}'$ is smooth. The coordinates $\{t_1, ..., t_d\}$ in Theorem \ref{semired} are uniquely determined up to permutation on $B$.  Let $H_i$ be the divisor of $\mathcal{Z}$ defined by
$$t_i = 0$$
for $i=1, ..., d$. Such divisors are uniquely determined locally near the central fibre $\mathcal{Z}_0= \pi_\cZ^{-1}(0)$. Since the general fibres of $\pi_\cZ$ are smooth, if we let $\cZ_{sing}$ be the singular set of $\cZ$, then $\pi_{\cZ}(\cZ_{sing} )$ is a proper subvariety of $B$. We choose a birational morphism $f: B'\rightarrow B$ such that $f^{-1}(\Psi(\cZ_{sing}))$ is a divisor with simple normal crossings. Therefore, we can always assume that $\pi_\cZ(\cZ_{sing})$ is a divisor of simple normal crossings in $B$.

\begin{lemma} \label{logca} The log pair $(\mathcal{Z}, K_{\mathcal{Z}} + \sum_{i=1}^d H_i)$ is a log canonical pair. Furthermore, $\cZ$ has at worst canonical singularities.

\end{lemma}

\begin{proof} Let $K$ be any irreducible and reduced component of the central fibre $\mathcal{Z}_0$. Recall an irreducible subvariety $V$ of a log pair $(X, \Delta)$ is called the log canonical center if there exists a birational morphism $f: Y \rightarrow X$ and a divisor $E$ of $Y$ such that $f(E)= V$ and the discrepancy $a(E, X, \Delta) =-1$.  Since $K$ is a complete intersection of $H_1, ..., H_d$ in a neighborhood of $K$ and the generic point of $K$ is smooth, it must be a log canonical center of $(\mathcal{Z}, \sum_{i=1}^d H_i)$ near $K$ by considering the  blow-up 
of the ideal defined by $\cap_{i=1}^d H_i$. On the other hand, $K$ is a component of the central fibre which is a semi-log canonical model. Therefore the normalization of $K$ is log canonical. We now can  apply the inverse adjunction of Hacon \cite{Ha} and so  $(\mathcal{Z}, K_{\mathcal{Z}} + \sum_{i=1}^d H_i)$ is log canonical near $K$. The first statement of the lemma then immediately follows. In fact, one can use the standard inverse adjunction formula of Kawakita \cite{Kaw} inductively by perturbing the divisor slightly, since $K$ is a complete intersection.

Let $\mathcal{F}: \cZ'  \rightarrow \cZ$ be a log resolution of $\cZ$. For any point $p\in \cZ$, we let $s=(s_1, s_2, ..., s_d)$ be the local holomorphic coordinates of $B$ such that
$$p\in \pi_\cZ^{-1}(\{s=0\})$$
and near $\pi_\cZ^{-1}(\{s=0\})$, we can assume 
$$\pi_\cZ (\cZ_{sing}) \subset \sum_{i=1}^d G_i, $$
where $G_i = \{ s_i=0\}$.
By the same argument before,  $(\cZ, \sum_{i=1}^d G_i)$ is a log canonical pair near the fibre $\pi_\cZ^{-1}(\{s=0\})$ and  near such a fibre, we have
$$K_{\cZ'} = \mathcal{F}^*K_\cZ + \sum_{i=1}^d (F^*G_i - G_i') + \sum_{i=1}^I a_i E_i + \sum_{k=1}^K b_k E'_k, $$
where $G'_i$ is the strict transform of $G_i$, the exceptional divisor of $\mathcal{F}$ consists of $\{E_j\}_{j=1}^J \cup \{E'_k\}_{k=1}^K$ with the prime divisor $E_j$ contained in $\sum_{i=1}^d \mathcal{F}^*G_i$ and the prime divisor $E'_k$ not contained in $\sum_{i=1}^d \mathcal{F}^*G_i$ for $i=1, ..., I$ and $k=1,..., K$.   Since $(\cZ, \sum_{i=1}^d G_i)$ is a log canonical pair, we have
$$a_i \geq -1, ~ b_k\geq -1.$$
By definition of $E_i$, we have 
$$\sum_{i=1}^d (F^*G_i - G_i') = \sum_{i=1}^I c_i E_i, ~ c_i \geq 1.$$
For each $k$, $\mathcal{F}(E_k')$ is not contained in $\sum_{i=1}^d G_i$ by definition, so the generic points of $\mathcal{F}(E_k')$ are smooth points of $\cZ$. This implies that $\mathcal{F}$ is a smooth blow-up near the generic points of $\mathcal{F}(E'_k)$  and so $b_k>0$. In summary, near the fibre $\pi_\cZ^{-1}(\{s=0\})$ we have
$$K_{\cZ'} = \mathcal{F}^*K_\cZ +  \sum_{i=1}^I (c_i+a_i) E_i + \sum_{k=1}^K b_k E'_k, $$
with
$$c_i+a_i\geq 0,~ b_k>0$$
 for $i=1, ..., I$ and $k=1, ..., K$.
This completes the proof of the lemma.

\end{proof}

%%%%%%%%%%%%%%%%%%%%%%%%%%%%%%%%%%%%%%%%%%%%%%%%%%

\section{Local $C^0$-estimate}

In this section, we will establish uniform $L^\infty$-estimates  for the local potentials of the K\"ahler-Einstein metrics for any fibres of $\pi_\cZ: \cZ \rightarrow B$ as in the diagram (\ref{diag3}). Such estimates were obtained in \cite{S2, S4} when $\dim B=1$, in which case the semi-stable reduction of Mumford is available. When $\dim B\geq 2$, we will make use of the (weakly) semi-stable reduction as in Theorem \ref{semired} together with an analytic form of Inverse of Adjunction, and then apply the maximum principle with suitable barrier functions for log canonical pairs $(\cZ, K_\cZ +\sum_{i=1}^d H_i)$.

We will first introduce some basic notations and quantities to set up the family of K\"ahler-Einstein equations for the fibres of $\pi_\cZ: \cZ \rightarrow B$. The following adapted volume measure is introduced in \cite{EGZ} to study the complex Monge-Amp\`ere equations on singular projective varieties.

\begin{definition} \label{adapvol} Let $X$ be a projective normal variety  with a $\mathbb{Q}$-Cartier canonical divisor $K_X$.  Then $\Omega$ is said to be an adapted measure on $X$ if  for any $z\in X$, there exists an open neighborhood $U$ of $z$ such that $$\Omega = f_U ( \alpha \wedge \overline{\alpha})^{\frac{1}{m}},$$ where $f_U$ is the restriction of a smooth positive function on the ambient space of a projective embedding $U$ and $\alpha$ is a local generator of the Cartier divisor $mK_X$ on $U$ for some $m\in \mathbb{Z}^+$.
\end{definition}
The adapted volume measure can also be defined for semi-log canonical models via pluricanonical embeddings. However, such volume measures are not locally integrable near the semi-log and log canonical singularities.

We now consider the following diagram from the stable family of smooth canonical models after base alternation as in the previous section.
\begin{equation}
\begin{diagram}
\node{\mathcal{X}'} \arrow{se,l}{ \pi' }  \arrow{e,t}{\Psi}   \node{\mathcal{Z}}   \arrow{s,r}{\pi_{\mathcal{Z}}}    \\
\node{}      \node{B}
\end{diagram}
\end{equation}
where $B$ is a unit ball in $\mathbb{C}^d$.  We denote $B^\circ$ the set of points over which the fibres are smooth. Since $K_{\mathcal{Z}/B}$ is $\pi_{\cZ}$-ample, we can assume there exist holomorphic sections $\eta_0, \eta_1, ..., \eta_N$ of $mK_{\cZ /B}$ for some sufficiently large $m$ such that they induce a projective embedding of $\cZ$
$$\Psi=[\eta_0, ..., \eta_N]: \cZ \rightarrow B\times \mathbb{CP}^N.$$
We let
$$\chi=\frac{1}{m} \ddbar \log\left(\sum_{i=0}^N |\eta_i|^2\right) $$ be the pullback of the Fubini-Study metric of $\mathbb{CP}^N$. In particular, for each $t\in B^\circ$,
\begin{equation}
\chi_t = \chi|_{\cZ_t} \in c_1(K_{\cZ_t}).
\end{equation}
is a smooth K\"ahler metric on $\cZ_t$.

For any $t\in B$,  $\eta_i|_{\cZ_t}$ is a holomorphic pluricanonical form on the regular part of $\cZ_t= \pi_\cZ^{-1}(t)$. We will define the relative volume form $\Omega$ by
\begin{equation}\label{volme1}
\Omega= \left( \sum_{i=0}^N |\eta_i|^2 \right)^{\frac{1}{m}}
\end{equation}
and for each $t\in B^\circ$,
\begin{equation}\label{volme2}
\Omega_t = \Omega|_{\cZ_t}
\end{equation}
is a smooth volume form on $\cZ_t$. By definition of $K_{\cZ/ B}$ and $\{\eta_0, ..., \eta_N\}$,
\begin{equation}
(\sqrt{-1})^{d}\wedge_{i=1}^d dt_i\wedge d\overline t_i \wedge \Omega
\end{equation}
is an adapted volume measure on $\cZ$. If we write the K\"ahler-Einstein metric $\omega_t$ by
$$\omega_t = \chi_t + \ddbar \varphi_t, $$
then the K\"ahler-Einstein equation on $\cZ_t$ for each $t\in B$ is given by
\begin{equation}\label{maeqn4}
(\chi_t + \ddbar \varphi_t)^n = e^{\varphi_t} \Omega_t.
\end{equation}

We fix a component $\mathbf{Z}$ of the central fibre $\mathcal{Z}_0$, where $\mathcal{Z}_0$ is a semi-log canonical model. Let $\mathbf{Z}'$ be the strict transform of $\mathbf{Z}$ by   
$\Psi$   in $\cX'$. 
 If we fix a smooth point $p\in \mathbf{Z}$, then $\mathbf{Z}'$ is locally defined by $x_{i_1}=x_{i_2}=...=x_{i_d}=0$ locally near $p'=\Psi^{-1}(p)$, where $t_1, ..., t_d, x_1, ..., x_{n+d}$ are local toric coordinates  in the semi-stable reduction (\ref{coordinates}) in Theorem \ref{semired}.  
 Without loss of generality, we assume that $i_1<i_2<...<i_d$ and let $D_j$ be the global  irreducible hypersurface defined by 
\begin{equation}
D_j= \{x_{i_j}=0 \}, ~~j=1, ..., d.
\end{equation}
Then $\mathbf{Z}'$ is a complete intersection of $\{ D_1, D_2, ..., D_d\}$ locally near $p$. When $d=1$, by the semi-stable reduction, $\mathbf{Z}'$ must be a smooth hypersurface itself and one can obtain local $C^0$-estimate as in \cite{S2}. However, when $d>1$, one does not expect $\mathbf{Z}'$ to be smooth in general as it might intersect itself and moreover, globally, $D_i$ might intersect $\mathbf{Z}'$. Our goal is to further apply toroidal blow-ups so that the strict transform of $\mathbf{Z}'$ after such blow-ups becomes a smooth complete intersection.

We can assume $D_1$, ..., $D_d$ are distinct after shrinking  $B$ so that $\mathbf{Z}'$ is a component of the complete intersection of $D_1$, ..., $D_d$.  In particular, we can assume each $D_i$ is an irreducible component of $\sum_{i=1}^d H_i'$, where $H'_i$ is the strict transform of $H_i$ by $\Psi$.

We will perform a sequence of  toroidal  blow-ups of  $\mathcal{X}'$ along locally toric subvarieties of $ \mathcal{X}'$ such that the resulting birational morphism
$$\Phi: \tilde \cX \rightarrow \mathcal{X}'$$
satisfies the following after possibly shrinking $B$.
 Let $\tilde D_1$, ..., $\tilde D_d$ and $\tilde{\mathbf{Z}}$  be the strict transform of $D_1$, ..., $D_d$ and  $\mathbf{Z}'$ respectively. We also let $\tilde{\mathcal{X}_0}$ be the strict transform of $\mathcal{Z}_0$ by $\Psi\circ\Phi$. 

 Then $\{ \tilde D_1, ..., \tilde D_d\}$ is a set of smooth divisors of simple normal crossing with 
 $$\cap_{i=1}^d \tilde D_i= \tilde {\mathbf{Z}}. $$   In particular, $\tilde D_i$ does not intersect $\tilde {\mathbf{Z}}$ and always has multiplicity $1$. The resulting family is given by the following diagram
\begin{equation}
\begin{diagram}
\node{\tilde \cX} \arrow{se,l}{ \tilde \pi }  \arrow{e,t}{\Phi}   \node{\cX' }   \arrow{s,r}{\pi'}  \arrow{e,t}{\Psi} \node{\cZ} \arrow{sw,r}{\pi_\cZ}  \\
\node{}      \node{B}
\end{diagram}
\end{equation}
For any closed point $q \in \tilde\cX$, there exist formal holomorphic toroidal coordinates $x_1, ..., x_{n+d}$ such that $\tilde \pi$ is given by near $p$
\begin{equation} \label{coordinates2}
t_i = \prod_{j=1}^{n+d} x_j^{\alpha_{i, j}}, ~~i=1, ..., d.
\end{equation}
Here $\alpha_{i, j}$ are nonnegative integers satisfying
$$\sum_{i=1}^d \alpha_{i, j}\geq 1 .$$
Furthermore,
$$(\sqrt{-1})^d\wedge_{i=1}^d dt_i \wedge d\overline{t_i} $$
is a strictly positive $(d,d)$-form away from a finite union of subvarieties of $\tilde\cX$ which does not contain any fibre.

We let $\tilde H_i$ be the strict transform of $H_i$ by $\Psi\circ\Phi$ and let $\{   E_j \}_{j=1}^J $ be the prime divisors  defined by
\begin{equation}\label{decodiv}
 \sum_{i=1}^d \tilde H_i = \sum_{j=1}^J  E_j + \sum_{i=1}^d\tilde D_i.
\end{equation}
Each $ E_j$ must be locally defined as the toroidal coordinate hyperplane and none of them lies in the exceptional locus of $\Psi\circ\Phi$. Let $\{E'_k\}_{k=1}^K$ be all the prime toroidal divisors of the exceptional locus of $\Psi\circ\Phi$, i.e. each $E'_k$ must be locally defined as the hyperplane from the toroidal coordinates and let 
$$\{ \tilde E_i \}_{i=1}^I= \{ E_1, E_2, ..., E_J, E'_1, E'_2, ... E'_K\}$$ with $I=J+K$.  We let $\tilde F_m$ be the prime divisors of exceptional locus of $\Psi\circ\Phi$ other than $E'_1, ..., E'_K$, for $m=1, ..., M$. $\tilde F_m$ might not be smooth or have simple normal crossings and there might be higher codimensional subvarieties in the exceptional locus of $\Psi\circ \Phi$. 
Then we have the following formula.

\begin{lemma} \label{adjun4} There exist $a_i \geq -1$, $b_m\geq 0$,  for $i=1, ..., I$ and $m=1, ..., M$   such that numerically
\begin{equation}
K_{\tilde \cX} + \sum_{i=1}^d \tilde D_i =   (\Psi\circ\Phi)^* (K_\cZ + \sum_{i=1}^d H_i ) +\sum_{i=1}^I  a_i  \tilde E_i + \sum_{m=1}^M b_m  \tilde F_m.
\end{equation}
\end{lemma}
\begin{proof} 
 It follows from Lemma \ref{logca} that $K_\cZ + \sum_{i=1}^d H_i$ is a log canonical pair and so 
\begin{equation}\label{adjun4forsec9}
K_{\tilde \cX} + \sum_{i=1}^d \tilde H_i =   (\Psi\circ\Phi)^* (K_\cZ + \sum_{i=1}^d H_i ) +\sum_{k=1}^K  c_k   E_K + \sum_{m=1}^M b_m  \tilde F_m 
\end{equation}
for some $c_j \geq -1$ and $b_m\geq -1$,  $j=1, ..., J$, $m=1, ..., M$. 
Immediately, we have 
$$K_{\tilde \cX} + \sum_{i=1}^d \tilde D_i =   (\Psi\circ\Phi)^* (K_\cZ + \sum_{i=1}^d H_i ) +\sum_{i=1}^I  a_i  \tilde E_i + \sum_{m=1}^M b_m  \tilde F_m$$
for $a_i\geq -1$ by the definition of $\tilde E_i$. 
 On the other hand, $\mathcal{Z}$ has canonical singularities and $\tilde F_m$ is not a component of $\sum_{i=1}^d \tilde D_i $ or $ (\Psi\circ\Phi)^* ( \sum_{i=1}^d H_i )$, therefor $b_m \geq 0$, for each $m=1, ..., M$. This proves the  lemma.

\end{proof}

We consider the following Monge-Amp\`ere equation on $\tilde \cX_t=\tilde \pi^{-1}(t)$ for each $t\in B^\circ$ by pulling back equation (\ref{maeqn4}) using $\Psi\circ\Phi$,
\begin{equation}\label{famke}
(\chi_t + \ddbar \varphi_t)^n = e^{\varphi_t} \Omega_t.
\end{equation}
Here for conveniences, we use the same notations for the pullback of $\chi$, $\Omega_t$ and $\varphi_t$ by $\Psi\circ\Phi$.

 Let $\Theta$ be a smooth volume form on $\tilde \cX$ near the central fibre and let $\omega$ be a fixed smooth K\"ahler form on $\tilde \cX$.
There exist $c_i > 0$  and an effective $\mathbb{Q}$-Cartier divisor $\mathcal{E}$ such that
$$(\Psi\circ\Phi)^*K_{\cZ } - \sum_{i=1}^I c_i [\tilde E_i]  - [\mathcal{E}]$$
is ample by Kodaira's lemma. We need to add the divisor $\mathcal{E}$ because the exceptional locus of $\Psi\circ \Phi $ might be larger than the support of $\cup_i^I \tilde E_i$. The support of $\mathcal{E}$ must contain $\sum_{m=1}^M \tilde F_m$.

Let $\sigma_{\tilde E_i}$  and $\sigma_\mathcal{E}$ be the defining section of $\tilde E_i$ and $\mathcal{E}$. There exist smooth hermitian metrics $h_{\tilde E_i}$  and $ h_\mathcal{E}$ equipped on the corresponding line bundle associated to the  divisor $\tilde E_i$ for $i=1, ..., I$ and $\mathcal{E}$  such that
\begin{equation}\label{kodlem}
 \chi + \sum_{i=1}^I c_i \ddbar\log h_{\tilde E_i} +  \ddbar \log h_\mathcal{E}>0
 \end{equation}
is a K\"ahler metric on $\tilde{\mathcal{X}}$ near the central fibre (note that $\chi$ is not strictly positive along the exceptional divisors).
%Furthermore, we can assume for each $i=1, ..., I$, $$c_i>0.$$
%
 We also let $\sigma_{\tilde D_i}$ be the defining sections for $\tilde D_i$ let $h_{\tilde D_i}$ be fixed smooth hermitian metrics on the line bundle associated to $\tilde D_i$ for $i=1, ..., d$. Without loss of generality, we can always assume the following holds on 
$\tilde{\mathcal{X}}$.
$$\sum_{i=1}^d  |\sigma_{\tilde D_i}|^2_{h_{\tilde D_i}} + \sum_{i=1}^I|\sigma_{\tilde E_i}|^2_{h_{\tilde E_i}}  + |\sigma_\mathcal{E}|^2_{h_\mathcal{E}} \leq 1.
$$
We notice that $\{\tilde E_i\}_{i=1}^I$ are divisors of simple normal crossings since they are locally defined by toroidal coordinates.

We will construct auxiliary conical K\"ahler metrics in the following lemma.
\begin{lemma} \label{conemet}

For sufficiently small $\varepsilon>0$,
\begin{equation}
\omega_\varepsilon =   \chi +  \varepsilon   \sum_{i=1}^I c_i \ddbar\log h_{\tilde E_i}  + \varepsilon \ddbar \log h_\mathcal{E} + \varepsilon^4    \sum_{i=1}^I  \ddbar |\sigma_{\tilde E_i}|^{2\varepsilon^2}_{h_{\tilde E_i}}
\end{equation}
is a smooth conical K\"ahler metric on $\tilde X$ near the central fibre $\tilde{\mathcal{X}_0}$ with cone angle $2\varepsilon^2 \pi $ along $\tilde E_i$.

\end{lemma}

\begin{proof} Let $\theta$ be a fixed smooth K\"ahler form on $\tilde \cX$. Since  $\{ \tilde E_i\}_{i=1}^I$ is a union of divisors of locally toric divisors, they have only simple normal crossings and for any sufficiently small $\varepsilon>0$
$$\theta+  \varepsilon^3  \sum_{i=1}^I  \ddbar |\sigma_{\tilde E_i}|^{2\varepsilon^2}_{h_{\tilde E_i}}  $$
 is a smooth conical K\"ahler metric with cone angle $ 2\varepsilon^2\pi$ along $\tilde E_i$, $i=1, ..., I$. This can be verified by straightforward local calculations. The lemma then follows by combining the above observation and  (\ref{kodlem}).

\end{proof}

Our goal is obtain a uniform upper bound for $\varphi_t$ and a uniform lower bound with suitable barrier functions. We begin by comparing different volume forms.

\begin{lemma} \label{volcom4} Let $\tilde{\mathbf{Z}}$  in $\tilde\cX$ be the strict transform of $\mathbf{Z}$ of $\cZ_0$ defined before. For sufficiently small $0<\varepsilon<<1$, there exists $c=c(\varepsilon)>0$ such that
\begin{equation}
 \frac{ \left( \wedge_{i=1}^d  dt_i\wedge d \bar{t_i} \right) \wedge (\omega_\varepsilon)^n }{ \left( \wedge_{i=1}^d  dt_i\wedge d \bar{t_i} \right) \wedge  \Omega_{t}  } \geq  c   \prod_{i=1}^I |\sigma_{\tilde{E}_i}|^{2(-a_i -1+\varepsilon^2 ) }_{h_{\tilde{E}_i}}
\end{equation}
on $\tilde\cX$ near the central fibre $\tilde{\mathcal{X}_0}$, where $a_i\geq -1$ is defined in Lemma \ref{adjun4} for $i=1, ..,, I$.

\end{lemma}

\begin{proof} By Lemma \ref{adjun4}, there exists a smooth volume form $\tilde\Omega$ on $\tilde \cX$ such that
\begin{equation}\label{volest4sec}
(\Psi\circ\Phi)^*\left( \frac{ (\sqrt{-1})^{d}\left( \wedge_{i=1}^d  dt_i\wedge d \bar{t_i} \right)\wedge \Omega_t}{|t_1 t_2 ... t_d|^2} \right) \leq \frac{  \prod_{i=1}^I  |\sigma_{\tilde E_i} |^{2a_i}_{h_{ \tilde E_i}} }{  \prod_{i=1}^n |\sigma_{\tilde D_i} |^2_{h_{\tilde D_i}} }  \tilde \Omega .
\end{equation}
for some  $a_j\geq -1$ as in Lemma \ref{adjun4}.

For any point $p\in \tilde \cX_0$ of the central fibre, we can pick local holomorphic (toric) coordinates $x_1, ..., x_{n+d}$  near $p$  such that  there exist nonnegative integers $\alpha_{i, j}$ such that
$$t_i = \prod_{j=1}^{n+d} x_j^{\alpha_{i, j}}, ~i=1, ..., d$$
as in (\ref{coordinates2}).

If $p$ is sufficiently close to $\tilde{\mathbf{Z}}$, we can assume that $\tilde{\mathbf{Z}}$  is locally defined by $x_{n+1}=x_{n+2}=...=x_{n+d}=0$ with $\tilde D_i$ defined by $x_{n+i}=0$, after rearrangement of $x_1, ..., x_{n+d}$ and so we can use $x_1, ..., x_n$ as coordinates for $\tilde{\mathbf{Z}}$.
In particular,
$$\sum_{i=1}^{d} \alpha_{i, j} = 1, ~ \alpha_{i, j} \geq 0, ~~ j= n+1, ..., n+d,$$
since for each $t_i$, there is one only and only one of $\{x_{n+1}, x_{n+2}, ..., x_{n+d}\}$ appearing as a factor in $t_i$
because $\tilde{\mathbf{Z}}$ is the completion intersection of $\tilde D_1, ..., \tilde D_d$. Without loss of generality, we can assume $$\alpha_{i, j} = \delta_{n+i, j}, ~ i=1, ..., d, ~j=n+1, ..., n+d. $$
%
%After possible rearrangement of $x_1, ..., x_n$, we can assume that the divisors defined by $\{x_i =0\}$ for $i=1, ..., I$ for some $I\leq n$ must be contained in  $\cup_{j=1}^J  E_j$ and the divisors defined by $\{ x_i=0\}$ for $i=I+1, ... , n$ must be contained in  $\cup_{j=1}^{J'} E'_j$.
%
%
For each $i=1, ..., d$, we have 
\begin{equation}
\frac{dt_i}{t_i} = \sum_{j=1}^{n+d} \alpha_{i, j} \frac{ dx_j}{x_j},
\end{equation}
and so
\begin{eqnarray*}
\frac{\sqrt{-1} dt_i\wedge d\overline{t_i}}{ |t_i|^2 } &=& \sum_{j=1}^{n+d} (\alpha_{i, j}  )^2\frac{ \sqrt{-1} dx_j\wedge d\overline{x_j} }{|x_j|^2} + \sum_{k=1}^{n+d} \sum_{l\neq k, l=1}^{n+d} \alpha_{i, k} \alpha_{i, l} \frac{ \sqrt{-1} dx_k\wedge d\overline{x_l} }{x_k \overline{x_l}}\\
&\geq & \frac{ \sqrt{-1} dx_{n+i}\wedge d\overline{x_{n+i}} }{|x_{n+i}|^2} + \sum_{k=1}^{n+d} \sum_{l\neq k, l=1}^{n+d} \alpha_{i, k} \alpha_{i, l} \frac{ \sqrt{-1} dx_k\wedge d\overline{x_l} }{x_k \overline{x_l}}
\end{eqnarray*}
There exists $c_\epsilon>0$ such that near $p$, we have
$$\omega_\varepsilon \geq c_\epsilon \sqrt{-1}\left( \sum_{i=1}^n \frac{dx_i \wedge d\overline{x_i}}{|x_i|^{2(1-\varepsilon^2)} } + \sum_{i=n+1}^{n+d} dx_i  \wedge d\overline{x_i} \right)  \geq c_\epsilon \sqrt{-1}\left( \sum_{i=1}^n \frac{dx_i \wedge d\overline{x_i}}{|x_i|^{2(1-\varepsilon^2)} } \right)
$$
and so
\begin{equation}\label{zclose}
\frac{(\sqrt{-1})^{d}\left( \wedge_{i=1}^d  dt_i\wedge d \bar{t_i} \right) \wedge \omega_\varepsilon^n}{|t_1 t_2 ... t_d|^2}  \geq \frac{ c^n (\sqrt{-1})^{n+d}\prod_{i=1}^{n+d}  d x_i \wedge d\overline{x_i}}{\prod_{i=1}^n |x_i|^{2(1-\varepsilon^2)} \prod_{i=n+1}^{n+d} |x_i|^2 } .
\end{equation}

  If $p$ is away from $\tilde{\mathbf{Z}}$, %
we can assume there are only $0\leq l <d $ members of $\{ \tilde D_1, \tilde D_2, ..., \tilde D_d \} $ that pass near $p$ and without loss of generality we can assume that they  correspond to the hyperplanes $x_{n+d-l+1}=0$, $x_{n+d-l+2}=0$, ..., $x_{n+d}=0$.  After rearrangement, we can assume that
$$ \alpha_{i, j} = \delta_{n+i, j}, ~ 1\leq i \leq d, ~n+d-l+1\leq j\leq n+d.$$
The differential of $t=(t_1, ..., t_d)$ is defined by
\begin{equation}\label{differential}
\frac{\partial(t_1, ..., t_d)}{\partial(x_1, ..., x_{n+d})} =\left[ \alpha_{i, j} \frac{\prod_{j=1}^{n+d} x_j^{\alpha_{i, j}}}{x_j} \right]_{1\leq i\leq d, 1\leq j\leq n+d} .
\end{equation}
and it has rank equal to $d$ away from the zeros of $x_j$. Therefore the rank of $[\alpha_{i, j}]_{1\leq i\leq d, 1\leq j\leq n+d}$ is equal to $d$ because the determinant of any $d\times d$ sub-matrix of (\ref{differential}) contains a factor as the determinant of the corresponding $d\times d$ submatrix of $[\alpha_{i, j}]_{1\leq i\leq d, 1\leq j\leq n+d}$. After re-arrangement, we can assume that
$$\mathcal{B}= [\alpha_{i, j}]_{ 1\leq i \leq d, ~n+1  \leq j \leq n+d }$$
is invertible and 
$$\left|\det \mathcal{B} \right|\geq 1$$
since $\det \mathcal{B} \in \mathbb{Z}$.
Straightforward calculations show that
$$
 \frac{(\sqrt{-1})^d \wedge_{i=1}^d  dt_i\wedge d \bar{t_i} }{|t_1 t_2 ... t_d|^2}  = I + II$$
with
$$I=     \frac{ \left(\det \mathcal{B}\right)^2 (\sqrt{-1})^d \prod_{j=n+1}^{n+d}dx_j\wedge d \bar{x_j } }{|x_{n+1} x_{n+2} ... x_{n+d}|^2} \geq \frac{  (\sqrt{-1})^d \prod_{j=n+1}^{n+d}dx_j\wedge d \bar{x_j } }{|x_{n+1} x_{n+2} ... x_{n+d}|^2}  $$
and $II$ not containing  $(\sqrt{-1})^d \prod_{j=n+1}^{n+d}dx_j\wedge d \bar{x_j }$.

Since hyperplanes defined by $x_1=0$, ..., $x_{n+d-l}=0$ must be among those of $\tilde E_i$, there exists $c_\epsilon>0$ such that
$$\omega_\varepsilon \geq c_\epsilon \sqrt{-1}\left( \sum_{j=1}^{n+d-l} \frac{dx_j \wedge d\overline{x_j}}{|x_j|^{2(1-\varepsilon^2 )} } + \sum_{j=n+d-1+1}^{n+d} dx_j \wedge d\overline{x_j} \right).
$$
Then immediately we have
\begin{equation}\label{zaway}
\frac{(\sqrt{-1})^{d} \left( \wedge_{i=1}^d  dt_i\wedge d \bar{t_i} \right) \wedge \omega_\varepsilon^n}{|t_1 t_2 ... t_d|^2}  \geq \frac{ (c_\epsilon)^{n+d-l} (\sqrt{-1})^{n+d} \prod_{j=1}^{n+d}  d x_j \wedge d\overline{x_j}}{\prod_{j=1}^{n+d-l} |x_j|^{2(1-\varepsilon^2 )} \prod_{i=n+d-l+1}^{n+d} |x_j|^2 }
\end{equation}

The estimates (\ref{zclose}) and (\ref{zaway}) hold near the central fibre $\tilde \cX_0$.  By shrinking $B$ slightly and covering $\tilde \cX$ by finitely many open sets where (\ref{zclose}) and (\ref{zaway}) hold,  there exists $c'_\epsilon>0$ such that   on $\tilde \cX$, we have
\begin{equation}\label{zaway3}
\frac{(\sqrt{-1})^{d} \left( \wedge_{i=1}^d  dt_i\wedge d \bar{t_i} \right) \wedge \omega_\varepsilon^n}{|t_1 t_2 ... t_d|^2}  \geq \frac{ c'_\epsilon ~\tilde\Omega}{ \left(\prod_{i=1}^I |\sigma_{\tilde E_i}|_{h_{\tilde E_i}}^{2(1-\varepsilon^2 )} \right) \left( \prod_{i=1}^n |\sigma_{\tilde D_i} |^2_{h_{\tilde D_i}} \right) }
\end{equation}
The lemma follows immediately from (\ref{zaway3}) and (\ref{volest4sec}).

\end{proof}

We will now estimate the lower bound of $\varphi_t$ in equation (\ref{famke}).
where $\chi_{t}$, $\Omega_{t}$ and $\varphi_{t}$ are pullback of $\chi_t$, $\Omega_t$ and $\varphi_t$ by $\Psi'$.
\begin{lemma} \label{54}  For any $\delta>0$, there exists $C_\delta>0$ such that for all $t \in B^\circ$, we have
$$ \varphi_{t} \geq \delta \log   \left( |\sigma_\mathcal{E}|^2_{h_\mathcal{E}} \prod_{i=1}^I   |\sigma_{ \tilde E_i} |_{h_{ \tilde E_i}} \right) - C_\delta$$
on each fibre $\tilde \cX_t$.

\end{lemma}

\begin{proof}    For any $\varepsilon>0 $ satisfying the conclusion of Lemma \ref{conemet}, we define
\begin{equation}
\varphi_{t, \varepsilon} = \varphi_{t} - \left( \varepsilon  \left( \sum_{i=1}^I  c_i \log |\sigma_{\tilde E_i} |^2_{h_{\tilde E_i}} + \log |\sigma_\mathcal{E}|^2_{h_\mathcal{E}} \right)+  \varepsilon^4 \sum_{i=1}^I   |\sigma_{\tilde E_i} |^{2\varepsilon^2 }_{h_{\tilde E_i}} \right)
\end{equation}
and $\varphi_{t, \varepsilon}$ satisfies the following equation away from all singular fibres
\begin{equation*}
\frac{  \left( \wedge_{i=1}^d  dt_i\wedge d \bar{t_i} \right) \wedge ( \omega_\varepsilon + \ddbar \log \varphi_{t, \varepsilon } )^n}{\left( \wedge_{i=1}^d  dt_i\wedge d \bar{t_i} \right) \wedge \Omega_{t}} \\
= \left( |\sigma_\mathcal{E}|^2_{h_\mathcal{E}} \prod_{i=1}^I |\sigma_{\tilde E_i}|^{2 c_i}_{h_{\tilde E_i}} \right)^{\varepsilon}  e^{\varphi_{t, \varepsilon } +  \varepsilon^4   \sum_{i=1}^I  |\sigma_{\tilde E_i} |^{2\varepsilon^2 }_{h_{\tilde E_i}}  }
\end{equation*}
by multiplying  $ \wedge_{i=1}^d  dt_i\wedge d \bar{t_i} $ to both sides of equation (\ref{famke}).

Suppose $q_t$ is the minimal point of $\varphi_{t, \varepsilon}|_{\tilde \cX_t}$ on a smooth fibre $\tilde \cX_t$.
The maximum principle implies that at the point $q_t$, there exist $c=c (\varepsilon)>0$ such that
\begin{eqnarray*}
e^{\varphi_{t, \varepsilon}}
&\geq&  \frac{ e^{  - \varepsilon^4   \sum_{i=1}^I  |\sigma_{\tilde E_i} |^{2\varepsilon^2}_{h_{\tilde E_i}}  }\left( \wedge_{i=1}^d  dt_i\wedge d \bar{t_i} \right) \wedge  \omega_\varepsilon ^n}{  \left(  |\sigma_\mathcal{E}|^2_{h_\mathcal{E}}\prod_{i=1}^I |\sigma_{\tilde E_i}|^{2 c_i}_{h_{\tilde E_i}}\right)^{\varepsilon }\left( \wedge_{i=1}^d  dt_i\wedge d \bar{t_i} \right) \wedge \Omega_{t}} \\
&\geq& \frac{ c}{  |\sigma_\mathcal{E}|_{h_\mathcal{E}}^{ 2\varepsilon }\prod_{i=1}^I |\sigma_{\tilde E_i}|^{2(1 + a_i +\varepsilon c_i  -\epsilon^2) }_{h_{\tilde E_i}}   }\\
&\geq& c.
\end{eqnarray*}
The second inequality follows from Lemma \ref{volcom4}.
Therefore on the fibre $\tilde \cX_t$ where $q_t$ lies, we have
$$\varphi_{t} \geq  2\varepsilon   \sum_{i=1}^I \log |\sigma_{E_i} |^2_{h_{E_i}}  - \log C'. $$
for some uniform constant $C'=C'(\varepsilon)$. The lemma then immediately follows.

\end{proof}

We now will derive the upper bound for $\varphi_{t}$ by the same argument in \cite{S2, S4}.

\begin{lemma} \label{vol1} There exists $C>0$ such that
$$\sup_{ \tilde \cX} \log \left(  \frac{ \left( \wedge_{i=1}^d  dt_i\wedge d \bar{t_i} \right)   \wedge \chi_{t} ^n}{\left( \wedge_{i=1}^d  dt_i\wedge d \bar{t_i} \right)   \wedge \Omega_{t} }     \right) \leq C .$$

\end{lemma}

\begin{proof} We use a  trick similarly to that in \cite{EGZ}. By the choice of $\{ \eta_j \}_{j=0}^{N}$,
$$ (\sqrt{-1})^{n+d} dt_1 \wedge d\overline{t_1}\wedge...\wedge dt_d \wedge d\overline{t_d}  \wedge \Omega$$
is an adapted volume measure on $\mathcal{X}$.  Since $\mathcal{X}$ is normal, for any point $p \in \mathcal{X}$, we can embed an open neighborhood $U$ of $p$ in $\mathcal{X}$  into $\mathbb{C}^N$ by $i: U \rightarrow \mathbb{C}^N$ (for example, we can assume that $\eta_0$ does not vanish near $p$ and we can  localize the embedding by $(\eta_1/\eta_0, ..., \eta_{N}/ \eta_0)$. Then $\chi|_U $ extends to a smooth K\"ahler metric on $\mathbb{C}^N$ and is quasi-equivalent to the Euclidean metric $\hat\omega = \sqrt{-1} \sum_{i=1}^N dz_i \wedge d\overline z_i$. Hence there exists $C_1>0$ such that near $i(p)$
$$ (C_1 )^{-1} \chi^{n+d} \leq  \sum_{ 1\leq i_1 < i_2 < ...< i_{n+d} \leq N } (\sqrt{-1})^{n+d}\prod_{k=1}^{n+d} dz_{i_k}  \wedge d \overline z_{i_k} \leq C_1 \chi^{n+d}.$$
Since $(\sqrt{-1})^{n+d} \left( \wedge_{i=1}^d  dt_i\wedge d \bar{t_i} \right)\wedge \Omega$ is an adapted volume measure on $\mathcal{X}$, for any $1\leq i_1 \leq i_2\leq ... \leq i_{n+d}\leq N$, there exist smooth nonnegative functions $f_{i_1, ..., i_{n+d}}$ in $U$  such that
$$ i^* \left(  (\sqrt{-1})^{n+d}\prod_{k=1}^{n+d} dz_{i_k} \wedge d \overline{z}_{i_k} \right)  =    f_{i_1, ..., i_{n+d}} \left( (\sqrt{-1})^{d}\wedge_{i=1}^d  dt_i\wedge d \bar{t_i} \right) \wedge \Omega .$$
%
%$$\Omega = i^* \left( \sum_{ 1\leq i_1 < i_2 < ...< i_{n+1} \leq N } f_{i_1, ..., i_{n+1}} \prod_{k=1}^{n+1} dz_{i_k} \wedge d \overline{z}_{i_k} \right). $$
%
This implies that there exists $C_2>0$ such that
$$ \chi^{n+d} \leq C_2 (\sqrt{-1})^{d} \left(  \wedge_{i=1}^d  dt_i\wedge d \bar{t_i} \right)\wedge \Omega. $$
The lemma is proved since $\sqrt{-1} \sum_{i=1}^d dt_i\wedge d\overline{t_i}$ is bounded above by a multiple of $\chi$.

\end{proof}

We immediately can achieve the following uniform upper bound for the potential $\varphi_t$.

\begin{lemma} \label{5upb}There exists $C>0$ such that for all $t \in \left(\mathcal{S}'\right) ^\circ$,
$$\sup_{\tilde \cX_t} \varphi_{t} \leq C. $$
\end{lemma}

\begin{proof} We apply the maximum principle to equation
$$
 (\chi_t + \ddbar\varphi_t)^n = e^{\varphi_t} \Omega_t.
$$
on $\tilde \cX_t$ at the maximal point of $\varphi_{t}$ and obtain
\begin{eqnarray*}
 \sup_{\tilde \cX_t} \varphi_t &\leq & \sup_{\tilde \cX_t} \log\left(  \frac{\chi_t^n }{\Omega_t} \right) = \sup_{\tilde \cX_t} \log \left( \frac{dt_1 \wedge d\overline t_1\wedge...\wedge dt_d \wedge d\overline t_d \wedge \chi_t^n }{ dt_1 \wedge d\overline t_1\wedge...\wedge dt_d \wedge d\overline t_d \wedge \Omega} \right) \\
 &\leq& \sup_{\tilde \cX_t} \log \left( \frac{ dt_1 \wedge d\overline t_1\wedge...\wedge dt_d \wedge d\overline t_d \wedge \chi^n }{ dt_1 \wedge d\overline t_1\wedge...\wedge dt_d \wedge d\overline t_d\wedge \Omega} \right) .
 \end{eqnarray*}
The lemma easily follows from Lemma \ref{vol1}.

\end{proof}

Let $\mathcal{V}$ be the the exceptional locus of the birational morphism $\Psi\circ\Phi$ and let 
\begin{equation}
\mathcal{V}'=\overline{\mathcal{V} \setminus  \cup_{i=1}^I \tilde E_i}. 
\end{equation}
Then there exist divisors $\mathcal{E}_1, ..., \mathcal{E}_L$ of $\tilde{\mathcal{X}}$ such that
\begin{equation}
\mathcal{V}'=\cup_{l=1}^L \mathcal{E}_l.
\end{equation}
Let $\sigma_{\mathcal{E}_l}$ be the corresponding defining section of $\mathcal{E}_l$ and let $h_{\mathcal{E}_l}$ be fixed smooth hermitian metrics on the line bundle associated with $\mathcal{E}_l$ for $l=1, ..., L$.  We conclude the section by the following proposition by combining Lemma \ref{54} and Lemma \ref{5upb}.
\begin{proposition} \label{c0pro} There exists $C>0$ and for any $\delta>0$, there exists $C_\delta>0$ such that the potential function $\varphi$  satisfies the following estimate on $\tilde{\mathcal{X}}$
\begin{equation}\label{c0bd}
\delta \log    \left( \left( 
\sum_{l=1}^L |\sigma_{\mathcal{E}_l}|^2_{h_{\mathcal{E}_l}} \right)
\left( \prod_{i=1}^I   |\sigma_{ \tilde{E}_i} |_{h_{\tilde{E}_i}}
   \right) \right) - C_\delta \leq \varphi \leq C.
\end{equation}
\end{proposition}
We also remark that $\varphi$ is smooth on $\tilde {\mathcal{X}}' $ away from the singular fibres. In fact, one can further show that $\varphi$ is smooth away from singular set of all the special fibres in the next section. Proposition \ref{c0pro} can be improved by bounding $\varphi$ below by a multiple $-\log(-\log)$-poles along the log and semi-log locus of special fibres and such improvement can be applied to prove that the local potentials of the Weil-Petersson current in Theorem \ref{main3} are not only bounded but continuous \cite{SSW}.

Finally, we would like to point out that the $C^0$-estimate of Proposition \ref{c0pro} still holds for stable families of canonical models of general type with log terminal singularities and it can also be generalized to stable families of semi-log canonical models with suitable technical modifications.

%%%%%%%%%%%%%%%%%%%%%%%%%%%%%%%%%%%%%%%%%%%%%%

\section{Analytic estimates and convergence}

In this section, we will obtain some standard higher order estimates for $\varphi_t$ from equation (\ref{famke}). Most argument are identical to estimates in \cite{S4} and we will only give necessary statements and sketch of proof. We will use the same notations in the previous section and let $\tilde{\mathbf{Z}}$ be the component of the central fibre $\tilde{\mathcal{X}}_0$ as the strict transform of a fixed component of $\mathcal{Z}_0$ by $\Psi\circ\Phi$.

Let $\omega_{t} = \chi_{t}+ \ddbar \varphi_{t}$ be the K\"ahler-Einstein metric on $\tilde \cX_t$ for $t\in B\circ$.
We define the barrier function
\begin{equation}
F= \left( 
\sum_{l=1}^L |\sigma_{\tilde{\mathcal{E}_l}}|^2_{h_{\tilde{\mathcal{E}_l}}} \right)
 \prod_{i=1}^I   |\sigma_{ \tilde{E}_i} |^2_{h_{\tilde{E}_i}}
\end{equation}
on $\tilde{\mathcal{X'}}$ over $B$. $F$ vanishes on every component of $\tilde{\mathcal{X}}_0$ except $\mathcal{\mathbf{Z}}$. The following lemma is an analogue of the Schwarz lemma with barrier functions.

\begin{lemma} \label{sch1} Then for any $\varepsilon>0$, there exists $C_\varepsilon>0$ such that for all $t \in B^\circ$, we have on $\tilde \cX_t$
\begin{equation}
\omega_{t} \geq  C_\epsilon   \left(F|_{\tilde X'_t}\right)^{\epsilon}   \chi_{t} .
\end{equation}
 \end{lemma}

\begin{proof} The proof follows immediately by applying the maximum principle to the following quantity
$$\log tr_{\omega_t}(\chi_t) + \epsilon \log F  -A \varphi_t$$  on $\tilde \cX_t$ for sufficiently large $A>0$ because $\varphi_t$ has vanishing Lelong number.

\end{proof}

\begin{lemma} \label{55} Let $\mathcal{G}$ be the set defined by $\mathcal{G} =\{F=0\}$. For any $k>0$ and any compact set $\mathcal{K}\subset\subset \tilde{\mathcal{X}} \setminus \mathcal{G}$, there exists $C_{k,K}>0$ such that for all $t\in B^\circ$,
\begin{equation}
||\varphi_t||_{C^k( \mathcal{K}\cap \tilde{\mathcal{X}}_t, \chi_{t})} \leq C_{k,\mathcal{K}}
\end{equation}
and so
\begin{equation}
||\omega_t||_{C^k(\mathcal{K} \cap \tilde \cX_t, \chi_{t})} \leq C_{k,\mathcal{K}}.
\end{equation}

\end{lemma}
\begin{proof}  For any point $p \in K\cap \tilde \cX_0$, for sufficiently small $t$, there exist coordinates $y=(y_1, y_2, ..., y_n)$ such that $(\tilde{\mathcal{X}})_{t}$ can be locally parametrized by $y$ and $\chi_{t}$ is uniformly equivalent  to  $dy \wedge d\bar y$. By Lemma \ref{sch1} and the original complex Monge-Amp\`ere equation (\ref{famke}) for $\varphi_{t}$, $\omega_t$ is uniformly bounded above and below with respect to $\chi_{t}$ away from $\mathcal{G}$, where $\pi$ is also nondegenerate. Then standard Schauder estimates and the linear estimates after linearizing the Monge-Amp\`ere equation (\ref{famke})  can be established locally, which gives uniform higher order regularity for $\varphi_t$.

\end{proof}

For any sequence $t_j \rightarrow 0 \in B^\circ$, by the uniform estimates for $\varphi_{t}$ away from $\mathcal{G}$, after passing to a subsequence, $\varphi_{t_j}$ converges smoothly to a smooth function $\varphi_0 $ on $\tilde{\mathcal{X}}_0\setminus \mathcal{G}|_{\tilde{\mathcal{X}}_0}$, a Zariski open dense subset of $\mathcal{Z}_0$. Furthermore, $\varphi_0$ satisfies the following conditions.

\begin{enumerate}

\item There exists $C>0$ such that $$ \sup_{\mathcal{Z}_0} \varphi_0 \leq C. $$

\item  $\varphi_0 \in \textnormal{PSH}(\cZ_0, \chi|_{\cZ_0})$.

\medskip

\item $\varphi_0$ has vanishing Lelong number everywhere.

\medskip

\item $\varphi_0$ solves the following equation on $\tilde{\mathcal{X}}_0\setminus \mathcal{G}|_{\tilde{\mathcal{X}}_0}$
$$(\chi_0 + \ddbar \varphi_0)^n = e^{\varphi_0}\Omega_0, $$

\end{enumerate}

By the uniqueness in Theorem \ref{song}, $\varphi_0$ must coincide with the unique solution constructed in Theorem \ref{song} (cf. \cite{S4}). Hence we have established the following lemma.

\begin{lemma} Let $\omega_t = \chi_t + \ddbar \varphi_t$ be the canonical K\"ahler-Einstein current on $\mathcal{X}_t$, $t\in B$ 
with $$(\chi_t + \ddbar \varphi_t)^n = e^{\varphi_t} \Omega_t.$$
Then $\varphi_t$ converges smoothly on $\tilde{\mathbf{Z}}_0 \setminus \mathcal{G}|_{\tilde{\mathcal{X}}_0}$ to a unique a unique $\varphi_0 \in \textnormal{PSH}(\mathcal{Z}_0, \chi_0)\cap C^\infty(\mathcal{R}_{\mathcal{Z}_0}) \cap L^\infty_{loc}(\mathcal{Z}_0\setminus \textnormal{LCS}(\mathcal{Z}_0))$ as $t\rightarrow 0$,  $t\in B^\circ$, where $\mathcal{R}_{\mathcal{Z}_0}$ is the smooth part of $\mathcal{Z}_0$ and $\textnormal{LCS}(\mathcal{Z}_0)$ is the non-log terminal locus of $\mathcal{Z}_0$.

\end{lemma}

The following lemma establishes the uniform non-collapsing condition for the K\"ahler-Einstein manifolds $(\mathcal{X}_t, g_t)$, for all $t\in B^\circ$ (c.f. Lemma 4.6 in \cite{S4}).

\begin{corollary} \label{noncol} Let $\mathbf{Z}$ be a component of $\mathcal{Z}_0$. There exist a point $p_0$ in the smooth part of $\mathbf{Z}$ , a smooth section $p(t): B \rightarrow \mathcal{Z}$ with  $p(0) =p_0$ and a constant $c >0$ such that for all $t\in B^\circ$,
\begin{equation}
Vol_{g_{t}} (B_{g_{t}}(p(t), 1)) \geq c,
\end{equation}
where $B_{g_{t}}(p(t), 1)$ is the unit geodesic ball centered at $p(t)$ in $(\tilde{\mathcal{X}}_{t}, g_{t})$.

\end{corollary}

All of the results above can be applied to any component of $\mathcal{Z}_0$ and we can always alter the base without changing any fibre so that we can assume $\mathcal{X}=\mathcal{Z}$.

%%%%%%%%%%%%%%%%%%%%%%%%%%%%%%%%%%%%%%%%%%%%%%

\section{Geometric estimates and convergence}

This section is a generalization of geometric estimate and convergence  in \cite{S2, S4}, where the base of the stable family has dimension $1$.  Let $\mathcal{X} \rightarrow B $ be a stable family of smooth canonical models over a unit Euclidean ball $B\subset \mathbb{C}^d$. Suppose the central fibre $\mathcal{X}_0$ is a semi-log canonical model
$$\mathcal{X}_0= \bigcup_{\alpha=1}^\mathcal{A} X_\alpha,$$
where each $X_\alpha$ is an irreducible component of $\mathcal{X}_0$.  Let $g_t$ be the unique K\"ahler-Einstein metric on $\mathcal{X}_t$ for $t\in B^\circ$. For any sequence $t_j\rightarrow 0$ with
$$t_j \in B^\circ,$$

We  state the main result of this section.

\begin{theorem}\label{main6.4}  Let $\mathcal{X} \rightarrow B$ be a stable family of $n$-dimensional  canonical models over a unit ball $B\subset \mathbb{C}^d$ as defined in Definition \ref{stabfib}. For any sequence $t_j \rightarrow 0$ with $t_j \in B^\circ$, after passing to a subsequence, there exist $m\in \mathbb{Z}^+$ and a sequence of $P_{t_j} = (p_{t_j, 1}, p_{t_j, 2}, ..., p_{t_j, m}) \in \amalg_{k=1}^m \mathcal{X}_{t_j}$, such that $(\mathcal{X}_{t_j}, J_{t_j}, g_{t_j}, P_{t_j})$ converge in pointed Gromov-Hausdorff distance to a finite disjoint union of metric spaces
$$(\mathcal{Y}, d_{\mathcal{Y}})  = \amalg_{k=1}^m (\mathcal{Y}_k, d_k)$$
satisfying the following.

\begin{enumerate}

\item For each $k$, the singular set $\mathcal{S}_k$ of $(\mathcal{Y}_k, d_k)$ is closed of Hausdorff dimension no greater than $2n-4$. $J_j$ converges smoothly to a complex structure $\mathcal{J}_k$ on $\mathcal{Y}_k \setminus \mathcal{S}_k$ and $\mathcal{Y}_k \setminus \mathcal{S}_k$ is a  K\"ahler manifold of complex dimension $n$. Furthermore, $g_j$ converge smoothly to a K\"ahler-Einstein metric $\mathpzc{g}_k$ on $(\mathcal{Y}_k \setminus \mathcal{S}_k, \mathcal{J}_k)$.

\medskip

\item The complex structure $\mathcal{J}_k$ on $\mathcal{Y}_k|_{\mathcal{Y}_k \setminus \mathcal{S}_k}$ uniquely extends to $\mathcal{Y}_k$ for each $k$.
\medskip

\item  $\cup_{k=1}^m(\mathcal{Y}_k, \mathcal{J}_k)$ is biholomorphic to $\mathcal{X}_0\setminus \textnormal{LCS}(\mathcal{X}_0) $. Furthermore, $\amalg_{k=1}^m \mathpzc{g}_k$ extends to  the unique K\"ahler current on $\mathcal{X}_0$.
\medskip

\item $\sum_{k=1}^m \textnormal{Vol}(\mathcal{Y}_k, d_k)= (K_{\mathcal{X}_t})^n$ for all $t\in B^\circ$.

\end{enumerate}

\end{theorem}

In fact, $m$ is the number of components of $\mathcal{X}_0$. The rest of the section is devoted to the proof of Theorem \ref{main6.4} based on the work in \cite{S4} (c.f. section 5, 6, 7 in \cite{S4}). 

Lemma \ref{55} implies that the   K\"ahler-Einstein metrics $g_{t_j}$ converge smoothly  to the unique canonical K\"ahler-Einstein metric on the central fibre $\mathcal{X}_0$ on a Zariski open set $\mathcal{U} \subset \mathcal{R}_{\mathcal{X}_0}$ of $\mathcal{X}_0$. We can pick  $\mathcal{A}$-tuple of nonsingular points as in the previous section such that
$$(p_0^{1}, p_0^{2}, ..., p_0^{\mathcal{A}}), ~~p_0^{\alpha} \in X_\alpha \cap \mathcal{U}, ~\alpha=1, ..., \mathcal{A},$$
where $\mathcal{A}$ is the number of the components of $\mathcal{X}_0$.

Let $(p_{t_j}^{1}, p_{t_j}^{2}, ..., p_{t_j}^{\mathcal{A}})$  a  sequence of $\mathcal{A}$-tuples of points $\in \mathcal{X}_{t_j}$ with $t_j \rightarrow 0$ such that
$$(p_{t_j}^{1}, p_{t_j}^{2}, ..., p_{t_j}^{\mathcal{A}})  \rightarrow (p_0^{1}, p_0^{2}, ..., p_0^{\mathcal{A}}) $$
 with respect to the fixed reference metric $\chi$ on $\mathcal{X}$.  We would like to study the Riemannian geometric convergence of $(\mathcal{X}_{t_j}, g_{t_j})$ as $t_j \rightarrow 0$.

\begin{lemma}  \label{GHcon}
After possibly passing to a subsequence, $(\mathcal{X}_{t_j},  g_{t_j}, (p_{t_j}^1, ..., p_{t_j}^\mathcal{A}))$ converges in pointed Gromov-Hausdorff topology to a metric length space
$$(\mathbf{Y}, d_{\mathbf{Y}})= \coprod_{\beta=1}^\mathcal{B} (Y_\beta, d_\beta)$$
as a disjoint union of  metric length spaces $(Y_\beta, d_\beta),$
satisfying the following. 
\begin{enumerate}

\item $\mathbf{Y} = \mathcal{R}_{\mathbf{Y}} \cup \mathcal{S}_{\mathbf{Y}}$, where $\mathcal{R}_{\mathbf{Y}}$ and $\mathcal{S}_{\mathbf{Y}}$ are the regular and singular part of $\mathbf{Y}$.
 $\mathcal{R}_{\mathbf{Y}}$ is an open K\"ahler manifold  and $\mathcal{S}_{\mathbf{Y}}$ is closed of Hausdorff dimension no greater than $2n-4$.

\medskip

\item $g_{t_j}$ converge smoothly to a K\"ahler-Einstein metric $g_{KE}$ on $\mathcal{R}_{\mathbf{Y}}$. \medskip

\item $\mathcal{R}_{\mathcal{X}_0}$ is an open dense set in $(\mathbf{Y}, d_{\mathbf{Y}})$ and
$$\mathcal{R}_{\mathcal{X}_0} \subset \mathcal{R}_{\mathbf{Y}}.$$
In particular, $g_{KE}$ coincides with the unique K\"ahler-Einstein current constructed in Theorem \ref{song} on $\mathcal{Z}_0$.

\medskip

\item  $\mathcal{B} \leq \mathcal{A}$ and
$$ \textnormal{Vol}(\mathbf{Y}, d_{\mathbf{Y}}) = \sum_{\beta=1}^\mathcal{B} \textnormal{Vol}(Y_\beta, d_\beta) = \textnormal{Vol}(\mathcal{X}_t, g_t) $$
for all $t\in B^\circ$.

\end{enumerate}

\end{lemma}

\begin{proof} The non-collapsing result in Corollary \ref{noncol} is the key ingredient to derive pointed Gromov-Hausdorff convergence for $(\mathcal{X}_{t_j}, g_{t_j})$. The proof of the lemma is identical to the proof of Lemma in \cite{S4} except for the statement (3) because we only derive $C^0$-convergence as well as $C^\infty$-convergence for $\varphi_t$ on $\mathcal{U}\subset \mathcal{R}_{\mathcal{X}_0}$, a Zariski open dense of $\cX_0$, by Lemma \ref{55} while such convergence is established everywhere on the regular part of the central fibre in \cite{S4} due to the fact that the birational morphism only takes place in the central fibre for semi-stable reduction over one-dimensional base. Now we will prove (3). First, it is obvious that
$$\mathcal{U} \subset \mathcal{R}_{\mathbf{Y}} $$
and $g_{t_j}$ converges smoothly to $g_{KE}$ on $ \mathcal{U}$.  Since $\mathcal{X}_0\setminus \mathcal{U}$ is a closed subvariety of $\cX_0$ of Hausdorff dimension no greater than $2n-2$  and $g_{KE}$ is smooth on $\mathcal{U}$, for any $p \in \mathcal{R}_{\cX_0}$, then there exists sufficiently small $\delta>0$ such that $B_{g_{KE}}(p, \delta)\cap \mathcal{U}$ is almost geodesically convex in $B_{g_{KE}}(p, \delta) \subset \cX_0$, i.e. for any point $p_1, p_2\in B_{g_{KE}}(p, \delta)\cap\mathcal{U}$ and $\epsilon>0$, there exists a smooth path $\gamma\in B_{g_{KE}}(p, \delta)\cap \mathcal{U}$ such that $$\mathcal{L}_{g_{KE}}(\gamma) \leq d_{g_{KE}}(p_1, p_2) + \epsilon. $$
Therefore the metric completion of $B_{g_{KE}}(p, \delta)\cap \mathcal{U}$ with respect to $g_{KE}$ coincides with its metric completion in $(\mathbf{Y}, d_{\mathbf{Y}})$ and so the entire $B_{g_{KE}}(p, \delta)$ is indeed isometrically embedded in $(\mathbf{Y}, d_{\mathbf{Y}})$. Therefore $p \in \mathcal{R}_\mathbf{Y}$ and
$$\mathcal{R}_{\cX_0} \subset \mathcal{R}_{\mathbf{Y}}. $$
We can furthermore conclude that $\varphi_{t_j}$ converges smoothly (with fixed gauge) to $g_{KE}$ on $\mathcal{R}_{\cX_0}$. The rest of the proof follows from the argument in section 5 of \cite{S4}.

\end{proof}

By applying Lemma \ref{GHcon}, we can improve Lemma \ref{sch1}  by removing the barrier function at the central fibre of $\cX$ by using the same argument as in \cite{S4}.
\begin{lemma}
Let $\omega_0=  \chi_0+ \ddbar \varphi_0$ be the unique K\"ahler-Einstein current on $\mathcal{X}_0$.  For any compact set $K \subset\subset \cX_0\setminus \textnormal{LCS}(\cX_0) $ in the central fibre $\cX_0$ of the original stable family $\cX$, there exists $c=c(K) >0$ such that
$$\omega_0 \geq c \chi_0 $$ on $K\cap \mathcal{R}_{\mathcal{X}_0}$.

\end{lemma}

 The following proposition can be proved by similar arguments in \cite{T1, DS1} as local $L^2$-estimates from Tian's proposal for the partial $C^0$-estimates. We let $h_t = ((\omega_t)^n)^{-1}= (e^{\varphi_t}\Omega_t )^{-1}$ be the hermitian metric on $\mathcal{X}_t$ for $t\in B^\circ$, where $\omega_t$ is K\"ahler-Einstein form associated to the Kahelr-Einstein metric $g_t$ on $\mathcal{X}_t$, $\Omega_t$ and $\varphi_t$ are defined in Section 4.

\begin{lemma} \label{l41}   Let $p_0 \in \mathcal{R}_{\mathcal{X}_0}$ and $p: B \rightarrow \mathcal{X}$ be a smooth section with $p(0)=p_0$.  For any $R>0$, there exists  $K_R >0$ such that if  $\sigma \in H^0(\mathcal{X}_t, mK_{\mathcal{X}_t})$ for $m\geq 1$ with $t\in B^\circ$, then
\begin{equation}\label{l51}
\|\sigma \|_{L^{\infty, \sharp}(B_{g_t}(p(t), R))} \leq K_R \| \sigma \|_{L^{2, \sharp}(B_{g_t}(p(t), 2R) )}
\end{equation}
\begin{equation}\label{l52}
 \|\nabla \sigma \|_{L^{\infty, \sharp}(B_{g(t)}(p(t), R)) } \leq K_R \|\sigma \|_{L^{2, \sharp}(B_{g_t}(p(t), 2R))},
\end{equation}
where $B_{g_t}(p(t), R)$ is the geodesic ball centered at $p(t)$ with radius $R$ in $(\mathcal{X}_t, g_t)$, the $L^2$-norms $||\sigma||_{L^{\infty,\sharp}}$ and  $||\nabla \sigma||_{L^{\infty,\sharp }}$ are defined with respect to the rescaled hermitian metric $(h_t)^m$ and the rescaled K\"ahler metric $mg_t$.

\end{lemma}

\begin{proof} For fix $R>0$, the Sobolev constant on $B_{g_t}(p(t), R)$ is uniformly bounded because of the Einstein condition and the uniform noncollapsing condition by Corollary \ref{noncol} for unit balls centered at $p(t)$. The proof follows by well-known argument of Moser's iteration on balls of relative scales using cut-off functions (c.f \cite{S2}).

\end{proof}

The following lemma gives a construction for global pluricanonical section on the limiting metric space $\mathbf{Y}$.

\begin{lemma} \label{ext1} Suppose $t_j \in B^\circ \rightarrow 0$ and $\sigma_{t_j} \in H^0(\mathcal{X}_t, mK_{\mathcal{X}_t})$ be a sequence of sections satisfying
$$\int_{\mathcal{X}_{t_j}} |\sigma_{t_j}|^2_{(h_{t_j})^m} dV_{mg_{t_j}} = 1. $$
Then after passing to a subsequence, $\sigma_{t_j}$ converges to a holomorphic section $\sigma$ of $mK_{\mathbf{Y}}$. Furthermore, the $\sigma|_{\mathcal{R}_{\cX_0}}$ extends to a unique $\sigma' \in H^0(\cX_0, mK_{\cX_0})$ and $\sigma$ vanishes along $\textnormal{LCS}(\cX_0)$.

\end{lemma}
The proof of lemma \ref{ext1} follows the same argument in Lemma 5.4 of \cite{S4}. The following is the local version of the partial $C^0$-estimate.
\begin{lemma} \label{par0}  Let $p_0\in \mathcal{R}_{\cX_0}$ and $p: B \rightarrow \mathcal{X}$ be a smooth section with $p(0)=p_0$.  For any $R>0$, there exist $m \in \mathbb{Z}^+$ and $c>0$ such that for any $t\in B^\circ$ and $q\in B_{g_t}(p(t), R)$, there exists $\sigma_t \in H^0(\cX_0,  mK_{\mathcal{X}_t})$ satisfying
%\
\begin{equation}
 |\sigma_t|^2_{ (h_t)^m} (q)\geq c,  ~~~ \int_{\mathcal{X}_t} |\sigma_t|^2_{(h_t)^m} dV_{mg_t} = 1.
\end{equation}

\end{lemma}

\begin{proof} The proof of the global partial $C^0$-estimate in \cite{DS1} can be directly applied here in the local case with the estimates in Lemma \ref{l41} and Lemma \ref{ext1} because the singular set of all iterated tangent cones of the Gromov-Hausdorff limit $(\mathbf{Y}, d_{\mathbf{Y}} )$ is closed and has Hausdorff dimension less than $2n-2$.

\end{proof}

Let $h_{\mathbf{Y}}$ be the hermitian metric on $K_\mathbf{Y}$ as the extension of $((\omega_{\mathbf{Y}})^n)^{-1}$ from $\mathcal{R}_{\mathbf{Y}}$, where $\omega_{\mathbf{Y}}$ is the K\"ahler-Einstein form on $\mathcal{R}_{\mathbf{Y}}$. We now pass the partial $C^0$-estimate in Lemma \ref{par0} to the limiting space $\mathbf{Y}$.

\begin{corollary} \label{par1} For any $p_0 \in \mathcal{R}_{\mathbf{Y}}$ and any $R>0$, there exist $m \in \mathbb{Z}^+$ and $c, C>0$ such that for any $q\in B_{d_{\mathbf{Y}}}(q, R)$, there exists $\sigma \in H^0(\mathbf{Y},  mK_{\mathbf{Y}})$ satisfying
%\
\begin{equation}
 |\sigma|^2_{ (h_{\mathbf{Y}})^m} (q)\geq c,  ~~~ \int_{\mathbf{Y}} |\sigma|^2_{(h_{\mathbf{Y}})^m} dV_{d_{\mathbf{Y}}} = 1.
\end{equation}

\end{corollary}
Such $\sigma$ has uniformly bounded gradient estimate in fixed geodesic balls and it can also be extended to a global pluricanonical section on $\mathcal{Z}_0$. There are many generalizations and variations of Lemma \ref{par0} and Corollary \ref{par1}. For example, if $p$ is a regular point in $\mathbf{Y}$, then there exist global pluricanonical sections $\sigma_0$, ..., $\sigma_n$ such that near $p$, $\sigma_0$ is nonzero and
$$\frac{\sigma_1}{\sigma_0}, ..., \frac{\sigma_n}{\sigma_0}$$
can be used as holomorphic local coordinates near $p$.  Lemma \ref{l41} and Corollary \ref{par1} are used to construct peak sections to separate distinct points on $\mathbf{Y}$.

With the partial $C^0$-estimate,   We can show that the regular part of the Gromov-Hausdorff limit coincides with the nonsingular part of $\mathcal{X}_0$.

\begin{lemma} \label{regularid}
Let $\mathcal{R}_{\mathbf{Y}}$ be the regular part of the metric space $(\mathbf{Y}, d_{\mathbf{Y}})$ and $\mathcal{R}_{\mathcal{X}_0}$ be the nonsingular part of the projective variety of $\mathcal{X}_0$. Then
$$\mathcal{R}_{\mathbf{Y}} = \mathcal{R}_{\mathcal{X}_0}$$
and they are biholomorphic to each other. Let $\mathcal{A}$ be the number of the component of $\mathcal{X}_0$ and $\mathcal{B}$ the number of the components of $\mathbf{Y}$.
 Then $$ \mathcal{A} = \mathcal{B}.$$
\end{lemma}

\begin{proof} The proof of the lemma is identical to that of Lemma 6.1 and Corollary 6.1 in \cite{S4} without modification.

\end{proof}

The following distance estimate is one of key estimate in \cite{S4} to identify the Riemannian geometric limit with the algebraic central fibre.
\begin{lemma} \label{fest} Let $X$ be a component of $\mathcal{X}_0$ and $p \in X \cap \mathcal{R}_{\mathcal{X}_0}$. Then the following hold.

\begin{enumerate}

\item For any $K \subset\subset \mathcal{X}_0 \setminus \textnormal{LCS} (\mathcal{X}_0)$, there exists $C_K>0$ such that for any $q\in X \cap \mathcal{R}_{\mathcal{X}_0}  \cap K$,
\begin{equation*}
d_{\mathbf{Y}}(p, q) \leq C_K.
\end{equation*}

\item For any $q_j \in \mathcal{R}_{\mathcal{X}_0}$ converging to some $q\in \textnormal{LCS} (\mathcal{X}_0)$ in $(\mathcal{X}_0, \chi_0)$, we have
\begin{equation*}
\lim_{j\rightarrow \infty} d_{\mathbf{Y}}(p, q_j) = \infty.
\end{equation*}

\end{enumerate}

\end{lemma}

\begin{proof} The proof is identical to that of Lemma 6.3 in \cite{S4}.

\end{proof}

In conclusion, in each component of $(\mathbf{Y}, d_\mathbf{Y})$, the local boundedness of the K\"ahler-Einstein potential is equivalent to the boundedness of distance in a uniform way. Now we can complete the proof of Theorem \ref{main6.4}  by applying combining the same argument in section 7 in \cite{S4} with the help of the above estimates.

%%%%%%%%%%%%%%%%%%%%%%%%%%%%%%%%%%%%%%%%%%%%%%%%%%%%%%%

\section{Proof of Theorem \ref{main1} }

\newcommand{\sO}{\mathscr{O}}

We will prove Theorem \ref{main1} in this section by using Theorem \ref{main6.4} in the previous section. The proof is given in the following steps.

First,  it follows from Matsusaka's big theorem that  all the manifolds in $\mathcal{K}(n, V)$ can be embedded in a fixed $\mathbb{CP}^N$ by the pluricanonical systems $|mK_X|$ for fixed sufficiently large $m \geq 1$ for any $X\in \mathcal{K}(n, V)$. Moreover, since $V=(K_X)^n$ is bounded, there are only finitely many possible Hilbert polynomials for such smooth canonical models by the Kollar-Matsusaka theorem. Let $\chi_k=h^0(m K_X)\in \QQ[m], 1\leq k \leq K$ be all possible Hilbert polynomials for all members in $
\mathcal{K}(n, V)$. We then may write
$$
\mathcal{K}(n, V)=\bigcup_{k=1}^K \mathcal{K}(\chi_k)
$$
where $\mathcal{K}(\chi_k) $ is the space of all $n$-dimensional smooth canonical models with Hilbert polynomial $\chi_k$.
Without loss of generality, we may assume that the sequence $\{X_i\}_{i=1}^\infty$ we start with in the assumption of Theorem \ref{main1} are contained in a single $\mathcal{K}(\chi)$ from now on with a fixed Hilbert polynomial $\chi$. In fact, it follows from the boundedness result of \cite{HMX} that all the possible stable limits of manifolds in $\mathcal{K}(\chi)$ also appear in a fixed projective space. We abuse the notations and still denote such a fixed projective space by $\mathbb{CP}^N=\mathbb{P} H^0(X,\mathscr{O}_X(mK_X))^\vee$ for a fixed and sufficiently divisible $m\gg 1$.

Let  $\hilb(\mathbb{CP}^N,\chi)$ denote the Hilbert scheme of $\mathbb{CP}^N$ with Hilbert polynomial $\chi$ and let $\hilb(X)\in\hilb(\mathbb{CP}^N,\chi)$ be the Hilbert point corresponding to the embedding $X\subset \mathbb{CP}^N$.  We introduce a locally closed subscheme as below
\begin{eqnarray*}
 \mathcal{H}^{\rm can}(\chi) = \{ \hilb(X)\in \hilb(\mathbb{CP}^N,\chi)   &|& X    \subset \mathbb{CP}^N \text{ is smooth}, \ \sO_{\mathbb{CP}^N}(1)|_X=\sO_X(mK_X), \\
 &&\chi(m)=h^0(\sO_{\mathbb{CP}}(m)) \}.\\
\end{eqnarray*}
By passing to a subsequence of $\{\hilb(X_i)\}_{i}^\infty\subset  \mathcal{H}^{\rm can}(\chi)$   if necessary we may assume that  $\{\hilb(X_i)\}_{i=1}^\infty \subset U$ with $U$ being  an  irreducible  component of
 $\mathcal{H}^{\rm can}(\chi)_{\rm red}\subset \mathcal{H}^{\rm can}(\chi)$, where $\mathcal{H}^{\rm can}(\chi)_{\rm red}$ is the reduction of $\mathcal{H}^{\rm can}(\chi)$. In particular, $U$ is a quasi-projective variety.  Let $\pi_U:\mathcal{X}_U\to U$ be the pull back of the universal family over  $\hilb(\mathbb{CP}^N,\chi)$.  Now by applying  Theorem 4.88 in \cite{K2} (cf. \cite{HX}), there is a projective, generically finite, dominant morphism $\phi: V\to U$ and a projective compactification $V\subset\overline{V}$ such that the pull-back
$\pi_V:\mathcal{X}_U\times_U V\to V$ extends to a stable family
\begin{equation*}
\xymatrix{
 \mathcal{X}_V:=\mathcal{X}_U\times_U V \ar@{>}[d]^{\pi_V} \ar@{^{(}->}[r] & \mathcal{X}_{\overline V}\ar@{>}[d]^{\pi_{\overline V}}\\
V  \ar@{^{(}->}[r]^{}  &  \overline V
}.
\end{equation*}
By our construction, $X_i$ corresponds to $\mathcal{X}_{\overline V, t_i}=\pi_{\overline V}^{-1}(t_i)$ for some $t_i \in  V$ and we let $g_{t_i}$ be the unique K\"ahler-Einstein metric on $X_i$. We can assume $t_i \rightarrow t_\infty \in \overline V$ after passing to a subsequence,. $\mathcal{X}_{\overline V, t_\infty}$ is a semi-log canonical model as the algebraic degeneration of $\{X_i\}_{i=1}^\infty$. We can now directly apply Theorem \ref{main6.4} and so $\{ (X_i, g_{t_i})\}_{i=1}^\infty$ converge
in pointed Gromov-Hausdorff topology to a complete K\"ahler-Einstein metric space biholomorphic to $\mathcal{X}_{\overline V, t_\infty} \setminus \textnormal{LCS}(\mathcal{X}_{\overline V, t_\infty})$. This completes the proof of Theorem \ref{main1}.

%%%%%%%%%%%%%%%%%%%%%%%%%%%%%%%%%%%%%%%%%%%%%%%%%%%%%%%

\section{Extension of the K\"ahler-Einstein current and proof of Theorem \ref{main2} }

We will prove Theorem \ref{main2} in this section. Let $\pi: \mathcal{X} \rightarrow \mathcal{S}$ be a stable family of canonical models as in Definition \ref{stabfib}. The following theorem was independently proved by Schumacher \cite{Sch} using pointwise differential calculations and Tsuji \cite{Ts} using dynamical construction of Bergman kernels.

\begin{theorem} \label{schum} $\pi|_{\mathcal{X}^\circ}: \mathcal{X}^\circ\rightarrow \mathcal{S}^\circ$ is a holomorphic family of smooth canonical models over a smooth variety $\mathcal{S}$. Let $\omega_t$ be the unique K\"ahler-Einstien metric on $\mathcal{X}_t$ for $t\in \mathcal{S}^\circ$ and $h$ be the hermitian metric on $K_{\mathcal{X}^\circ /\mathcal{S}^\circ}$ defined by
$$h_t = (\omega_t^n)^{-1}. $$
Then the curvature $\theta = -\ddbar\log h$ of $h$ is a smooth nonnegative closed $(1,1)$-form on $\mathcal{X}^\circ$. If $\pi|_{\mathcal{X}^\circ}: \mathcal{X}^\circ \rightarrow \mathcal{S}^\circ$ is nowhere infinitesimally trivial, then $\theta$ is strictly positive everywhere on $\mathcal{X}^\circ$.

\end{theorem}

We will apply Theorem \ref{main1} and particularly Proposition \ref{c0pro} to understand the singular behavior of $h$ near singular fibres and thus global positivity of $\theta$ on $\cX$.

For any point $t\in \cS$, there exists a sufficiently small neighborhood $U$ of $s$ in $\cS$ such that $K_{\cX_U/ U}$ is $\pi$-ample, where $\cX_U = \pi^{-1}(U)$.  Let $$\{\eta_0, ~\eta_1, ..., ~\eta_N \}$$ be a basis of holomorphic sections in the linear system $\left| mK_{\cX_U/ U} \right| $ that induces a projective embedding of $\cX_U$. Then we define the relative volume form $\Omega_{\cX_U}$ on $\cX$ such that on $U$,
$$\Omega_{\cX_U}=  \left(\sum_{i=0}^N |\eta_i|^2 \right)^{\frac{1}{m}}$$
and by a smooth partition of unity  on $\cS$, we can construct a a relative volume form $\Omega$ on $\cX$ and we let
$$\chi = \ddbar \log \Omega.$$

\begin{lemma} \label{loclel}Let $\pi: \mathcal{X} \rightarrow \mathcal{S}$ be a stable family of $n$-dimensional canonical models as in Definition \ref{stabfib}. Let $\Omega$ be the relative volume form defined as above. Let $\Omega_t = \Omega|_{\cX_t}$ and
\begin{equation}\label{fdef}
F=\log  \frac{(\omega_t)^n}{\Omega_t }
\end{equation}
on $\mathcal{X}^\circ$ for $t\in \mathcal{S}^\circ$.  Then exists $C>0$ such that
$$\sup_{\mathcal{X}^\circ} F < \infty $$
and on $\mathcal{X}^\circ$,
$$\chi + \ddbar F \geq 0.$$
Therefore $F$ extends uniquely to quasi-plurisubharmonic function in $\textnormal{PSH}(\mathcal{X}, \chi)$. Furthermore, the Lelong number of $F$ vanishes everywhere on $\mathcal{X}$.
\end{lemma}
\begin{proof} The K\"ahler-Einstein metric equation on $\cX_t$ is equivalent to the complex Monge-Amp\`ere equaiton
$$(\omega_t)^n = e^{\varphi_t} \Omega_t $$
for $\varphi_t$,
where $\omega_t = \chi|_{\mathcal{X}_t} + \ddbar \varphi_t$. Then
$$F|_{\mathcal{X}_t}= \varphi_t $$ is uniformly bounded above by Lemma \ref{5upb} as $\varphi_t$ is uniformly bounded above on $\mathcal{X}$. Using Theorem \ref{schum}, we have   the following direct calculation on $\mathcal{X}^\circ$,
$$\ddbar F = \ddbar \log \frac{\omega_t^n}{\Omega_t}= \textnormal{Ric}(h) - \chi \geq - \chi. $$
Therefore $F \in \textnormal{PSH}(\mathcal{X}^\circ, \chi)$ and it uniquely extends to $\mathcal{X}$ and $F\in \textnormal{PSH}(\mathcal{X}, \chi)$ since $\chi$ has locally bounded potentials. By Lemma \ref{54}, there exists a divisor $D$ of  $\mathcal{X}$ that does not contain any component of fibres of $\pi: \mathcal{X} \rightarrow \mathcal{S}$ such that for any $\epsilon>0$, there exists $C_\epsilon>0$ such that
$$F \geq \epsilon \log |\sigma_D|^2_{h_D} - C_\epsilon,$$
where $\sigma_D$ is a defining section of $D$ and $h_D$ a fixed smooth hermitian metric for the line bundle associated to $D$. By   definition, $F$ has vanishing Lelong number everywhere on $\mathcal{X}$. The nonnegative current $\theta= \textnormal{Ric}(h)$ on $\mathcal{X}^\circ$ extends to $\mathcal{X}$ and is given by
$$\theta = \chi+ \ddbar F \in c_1(K_{\mathcal{X}/\mathcal{S}}).$$
In particular, $\theta$ has vanishing Lelong number.

\end{proof}

We can simply let $\varphi = F$ since $F$ is the trivial extension of $\varphi_t$ fibrewise. The following regularization lemma is proved in \cite{BK}.

\begin{lemma}\label{kol} Let $(M, \omega)$ be a K\"ahler manifold equipped with a smooth K\"ahler form $\omega$. Assume $\gamma$ is a continuous $(1,1)$-form on $M$. If $\phi\in \textnormal{PSH}(M, \gamma)$ has vanishing Lelong number everywhere, then for any open bounded domain $U$
with $\overline{U}\subset\subset M$, there exist a sequence $\varepsilon_j \rightarrow 0^+$ and a sequence $\{ \phi_j \}_{j=1}^\infty \subset \textnormal{PSH}(M, \gamma + \varepsilon_j \omega)\cap C^\infty(M)$ such that $\phi_j$ decreasingly converge to $\phi$ in $U$.

\end{lemma}

\begin{corollary}\label{nef} $K_{\mathcal{X}/\mathcal{S}}$ is nef.

\end{corollary}

\begin{proof} Without loss of generality, we can assume $\mathcal{X}$ is smooth after applying resolution of singularities. For any closed curve $\mathcal{C}$ in $\mathcal{X}$, we consider a small open neighborhood $U$ of $\mathcal{C}$ in $\mathcal{X}$. For a fixed K\"ahler form $\theta$ on $\mathcal{X}$, there exist $\varepsilon_j \rightarrow 0^+$ and $\varphi_j \in \textnormal{PSH}(U, \chi + \varepsilon_j \theta) \cap C^\infty(U)$ by Lemma \ref{kol}  such that $\varphi_j $ converge to $\varphi$ decreasingly.
$$K_{\mathcal{X}/\mathcal{S}} \cdot \mathcal{C} =   \int_{\mathcal{C}} \chi = \lim_{j\rightarrow \infty} \int_{\mathcal{C}} (\chi+ \varepsilon_j \theta+ \ddbar \varphi_j) \geq 0. $$
This proves the corollary.

\end{proof}

\section{The Weil-Petersson metric and proof of Theorem \ref{main3}}

We will prove Theorem \ref{main3} in this section. 
We will keep the same notations in section 8 and let $\omega = \chi+ \ddbar \varphi$ as in (\ref{fdef}) and by Lemma \ref{loclel}, $\omega$ is a K\"ahler current on the total space $\mathcal{X}$ and is smooth on $\mathcal{X}^\circ$. In particular, $\varphi |_{\mathcal{X}_t} = \varphi_t $ and $\omega|_{\mathcal{X}_t} = \omega_t$ for $t\in \mathcal{S}^\circ$, where $\omega_t=\chi+ \ddbar \varphi_t$ is the unique K\"ahler-Einstein metric on $\mathcal{X}_t$ for $t\in \mathcal{S}^\circ$. It is well-known that the Weil-Petersson metric $\omega_{WP}$ on $\mathcal{S}^\circ$ is the curvature of the Deligne pairing by the relative canonical bundle and can be calculated as the push-forward of  $\omega^{n+1}$ as below (c.f. \cite{Sch})
\begin{equation}\label{wpfor1}
\omega_{WP} (t) = \int_{\mathcal{X}_t} \omega^{n+1}
\end{equation}
for $t\in \mathcal{S}^\circ$. 
Direct calculations give the following formula.
\begin{lemma} \label{wpref} On $\mathcal{X}^\circ$, the Weil-Petersson metric is given by
\begin{equation}
\omega_{WP}= \int_{\mathcal{X}_t} \chi^{n+1} + \ddbar \left( \int_{\mathcal{X}_t } \varphi \left(\sum_{k=0}^n\chi^k \wedge (\chi+ \ddbar \varphi )^{n-k} \right)      \right)
\end{equation}
for $t\in \mathcal{S}^\circ$.

\end{lemma}
%

%%%%%%%%%%%%%%%%%%%%%%%%%%%%%%%%%%%%%%%%%%

%For any point $t\in \cS$, there exists a sufficiently small neighborhood $U$ of $s$ in $\cS$ such that $K_{\cX_U/ U}=K_{\cX_U}$ is ample, where $\cX_U = \pi^{-1}(U)$. One can find an embedding  $$\iota: \cX_U \rightarrow \mathbb{CP}^N \times U,$$
%
%where $\mathcal{O}_{\mathbb{CP}^N}(1)|_{\cX_U} = mK_{\cX_U/U}$ for some sufficiently large $m\geq 1$.   Let $$\{\eta_0, ~\eta_1, ..., ~\eta_N \}$$ be a basis of holomorphic sections of $mK_{\cX_U/ U}$ that induces a projective embedding of $\cX_U$. Then we define the relative volume form $\Omega_{\cX_U}$ on $\cX$ such that on $U$,
%
%$$\Omega_{\cX_U}=  \left(\sum_{i=1}^N |\eta_i|^2 \right)^{\frac{1}{m}}$$
%
%
%such that
%
%$$\chi_U = \ddbar \log \Omega_{\cX_U}$$
%
%is the restriction to $\cX_U$ of the Fubini-Study metric on $\mathbb{CP}^N$.
%
%On $\cX_U$, $\chi$ and $\chi_U$ are related by the following
%
%$$\chi= \chi_U + \ddbar \psi, $$
%
%where $\psi$ is the restriction of a smooth quasi-plurisubharmonic function on $\mathbb{CP}^N\times U$ to $\cX_U$.

Before we start estimating the local potentials of the Weil-Petersson metric $\omega_{WP}$, we will derive some basic estimates.
We will consider a stable family of smooth canonical models after base alteration as in section 8.
\begin{equation}\label{ssrsec9}
\begin{diagram}
\node{\mathcal{X}'} \arrow{se,l}{ \pi' }  \arrow{e,t}{\Psi}   \node{\cZ:=\mathcal{X}\times_\mathcal{S} \mathcal{S}'}  \arrow{e,t}{f'} \arrow{s,r}{\pi_{\mathcal{Z}}}   \node{\mathcal{X} } \arrow{s,r}{\pi} \\
\node{}      \node{\mathcal{S}'} \arrow{e,t}{f}  \node{\mathcal{S}}
\end{diagram}
\end{equation}
We will further apply local toric blow-ups to $\mathcal{X}'$ near a central fibre as in the following diagram so that near $0\in B \subset  \mathcal{S}'$  the proper transformation of each fibre $\mathcal{Z}_t$ of $\mathcal{Z}$ is a union of disjoint smooth projective manifolds of dimension $n$. Furthermore, we can assume that the blow-up center does not contain any component of the proper transformation of each fibre $\mathcal{Z}_t$ of $\mathcal{Z}$, and the exceptional locus of $\Psi\circ\Phi$ is a union of divisors locally defined by the toric coordinates. 

\begin{equation} \label{ssr4}
\begin{diagram}
\node{\tilde \cX} \arrow{se,l}{ \tilde \pi }  \arrow{e,t}{\Phi}   \node{\cX' }   \arrow{s,r}{\pi'}  \arrow{e,t}{\Psi} \node{\cZ} \arrow{sw,r}{\pi_\cZ}  \\
\node{}      \node{B}
\end{diagram}
\end{equation}
For any closed point $q \in \tilde\cX$, there exist formal holomorphic toroidal coordinates $x_1, ..., x_{n+d}$ such that $\tilde \pi$ is given by near $p$
\begin{equation} \label{coordinates95}
t_i = \prod_{j=1}^{n+d} x_j^{\alpha_{i, j}}, ~~i=1, ..., d.
\end{equation}
%
%Here $\alpha_{i, j}$ are nonnegative integers satisfying
%
%$$\sum_{i=1}^d \alpha_{i, j}\geq 1 .$$
%
We remark that the fibres of $\tilde \pi$ no longer have the same dimensions as $\pi'$, but $\tilde \pi$ is still equi-dimensional over $(\mathcal{S}')^\circ \cap B$. For any exceptional divisor $D$ of $\Psi\circ\Phi$ and $t\in B$, we have 
\begin{equation} \label{vanishfibre}
\dim \left( \Psi\circ\Phi(D\cap \mathcal{Z}_t) \right)\leq n-1.
\end{equation}

The following lemma immediately follows from (\ref{vanishfibre}).

\begin{lemma} Let $D$ be an exceptional divisor of $\Psi\circ\Phi$. Then for any $t\in B$, we have 
\begin{equation}
\chi^n |_{D \cap \tilde{\mathcal{X}}_t} = 0.
\end{equation}

\end{lemma}

\begin{lemma}\label{divisorinte} Let $D$ be an exceptional divisor of $\Psi\circ\Phi$ such that the support of $D$ coincides with the exceptional locus of $\Psi\circ\Phi$ on $\tilde X$. Let $\sigma_D$ be a defining section of $D$ and $h_D$ a smooth hermitian metric of the line bundle associated to $D$. Then there exist $C>0$ such that for any $t\in B^o$, 
\begin{equation}
- \int_{\tilde{\mathcal{X}}_t} \left( \log|\sigma_D|^2_{h_D} \right) \chi^n \leq C. 
\end{equation}

\end{lemma}

\begin{proof} For any closed point $p \in \tilde\cX$, there exist local holomorphic toric coordinates $x_1, ..., x_{n+d}$ such that $\tilde \pi$ is given by near $p$
\begin{equation} \label{coordinates96}
t_i = \prod_{j=1}^{n+d} x_j^{\alpha_{i, j}}, ~~i=1, ..., d, ~\alpha_{i,j} \in \mathbb{Z}^+\cup\{0\}.
\end{equation}
The exceptional divisor is given by  $$\cup_{j=1}^{n+d-m} \{x_j=0\}$$
for some $d\leq m \leq n+d$. 
 The $d \times (n+d)$-matrix $[\alpha_{i, k}]$ must have rank $d$, otherwise $\chi^n$ will vanish everywhere on some open set of a smooth fibre of $\mathcal{Z}$. Recall that $\mathcal{F} = \Psi\circ\Phi: \tilde{\mathcal{X}} \rightarrow \mathcal{Z} \hookrightarrow \mathbb{CP}^N$ for some algebraic embedding of $\mathcal{Z}$ into a projective space $\mathbb{CP}^N$. Let $\mathcal{V}$ be an open neighborhood of $\mathcal{F}(p)$ in $\mathbb{CP}^N$ and for simplicity we can assume that $\chi$ is equivalent to the Euclidean metric on $\mathcal{V}$.
Therefore $\mathcal{F}=(\mathcal{F}_1, ..., \mathcal{F}_N)$ is holomorphic and 
$$\mathcal{F}^*\chi =  \sqrt{-1}  \sum_{\alpha=1}^N \sum_{i, j=1}^{n+d}\frac{\partial \mathcal{F}_\alpha}{\partial x_i }  \overline{ \frac{\partial \mathcal{F}_\alpha}{\partial x_j} }dx_i \wedge d\bar x_j = \sqrt{-1} \sum_{i, j=1}^{n+d} G_{i \bar j } dx_i \wedge d\bar x_j$$ 
and so
$$\mathcal{F}^*\chi^n = (\sqrt{-1})^n \sum_{1\leq i_1 < ... <i_n \leq n+d, 1\leq j_1 < ... <j_n \leq n+d} \det\left( G_{i\bar j} \right)_{i_1... i_n, \bar j_1 ...\bar j_n} dx_{i_1}\wedge d\bar x_{j_1} \wedge ... \wedge dx_{i_n}\wedge d \bar x_{j_n},$$
where $\left( G_{i\bar j} \right)_{i_1... i_n, \bar j_1 ...\bar j_n} $ is the  $(n\times n)$-submatrix of $\left( G_{i \bar j}  \right)$. 
In particular, $\mathcal{F}$ is holomorphic and algebraic in $x_1$, ..., $x_{n+d}$ and 
$\det\left( G_{i\bar j} \right)_{i_1... i_n, \bar j_1 ...\bar j_n} $ is an algebraic function in $x_1$, $\bar x_1$, ..., $x_{n+d}$, $\bar x_{n+d}$. 

Since $\log|\sigma_D|^2$ is locally the sum of $\log$ of every  exceptional prime divisor,   it suffices to consider $\log|x_1|^2$ instead of $\log |\sigma_D|^2_{h_D}$ near a fixed point $p\in \tilde{\mathcal{X}}$, where $\{x_1=0\}$ is an exceptional divisor of $\Psi\circ\Phi$.  We can always assume $\alpha_{1,1}>0$ after rearrangement in $t_1, ..., t_d$ because $\sum_{i=1}^d a_{i, 1}>0$, otherwise $\{x_1=0\}$ cannot be an exceptional divisor. 

It also suffices to prove the lemma locally near $p$. We let $p=0$ and 
$$\mathcal{U} =\{ 0\leq |x_1|, ..., |x_{n+d}| <1\}. $$
We write
$$\mathcal{F}^* \chi^n =\Theta= \Theta_1 + \Theta_2$$
such that 
$\Theta_1$ contains $d x_1$ or $d \bar x_1$ and $\Theta_2$ does not contain $d x_1$ or $d \bar x_1$. Obviously both $\Theta_1$ and $\Theta_2$ are real valued nonnegative $(n, n)$-forms on $\tilde{\mathcal{X}}$. Since $\Theta_1$ and $\chi^n$ vanish on $\{x_1=0\}$,   $\Theta_2$ must also vanish on $\{x_1=0\}$. 
Let
\begin{eqnarray*}
\Theta_1 &=&  (\sqrt{-1})^n \sum_{1= i_1 < ... <i_n \leq n+d, 2\leq j_1 < ... <j_n \leq n+d} \Theta_{i_1...i_n, j_1...j_n} dx_{i_1}\wedge d\bar x_{j_1}\wedge ... \wedge dx_{i_n}\wedge d\bar x_{j_n} \\
&&  +(\sqrt{-1})^n \sum_{2\leq i_1 < ... <i_n \leq n+d, 1= j_1 < ... <j_n \leq n+d} \Theta_{i_1...i_n, j_1...j_n} dx_{i_1}\wedge d\bar x_{j_1}\wedge ... \wedge dx_{i_n}\wedge d\bar x_{j_n}\\
&& +(\sqrt{-1})^n \sum_{1= i_1 < ... <i_n \leq n+d, 1= j_1 < ... <j_n \leq n+d} \Theta_{i_1...i_n, j_1...j_n} dx_{i_1}\wedge d\bar x_{j_1}\wedge ... \wedge dx_{i_n}\wedge d\bar x_{j_n},
\end{eqnarray*}
and
$$\Theta_2=  (\sqrt{-1})^n \sum_{2\leq i_1 < ... <i_n \leq n+d, 2\leq j_1 < ... <j_n \leq n+d} \Theta_{i_1...i_n, j_1...j_n} dx_{i_1}\wedge d\bar x_{j_1}\wedge ... \wedge dx_{i_n}\wedge d\bar x_{j_n},$$
where all the coefficients are bounded algebraic functions in $x_1$, $\bar x_1$, ..., $x_{n+d}$, $\bar x_{n+d}$.
The domain $\mathcal{U}$ can be covered by  subdomains as 
\begin{equation}\label{subdo}
\mathcal{U}_{i_2, i_3, ... ,i_{n+d}} = \{ |x_1|\leq 1, |x_{i_2}| \leq |x_{i_3}| \leq ... \leq |x_{i_{n+d}}| \leq 1\},
\end{equation}
where $(i_2, ..., i_{n+d})$ is a permutation of $(2, 3, ..., n+d)$.   It suffices to prove the lemma on one of such  subdomains and without loss of generality, we can choose $\mathcal{U}' = \mathcal{U}_{2,3,...,n+d}$ for $i_2=2$, ..., $i_{n+d}=n+d$.

We will perform the standard row reduction for the matrix $[a_{i, j}]_{1\leq i \leq d, 1\leq j \leq n+d}$ with the resulting matrix $[\alpha_{i, j}]_{1\leq i \leq d, 1\leq j \leq n+d}$ of the following type

$$[\beta_{i, j}]_{1\leq i \leq d, 1\leq j \leq n+d}= 
\begin{pmatrix}
*...*  & * ... * & * ... * &  * ... * & ... & * ... *\\
0 ... 0&* ... * & * ... * &  * ... * & ... & *... * \\
0 ... 0 &0 ...0 & * ... * &  * ... * & ... & *... * \\
0 ... 0 &0 ...0 & 0 ...0 &  * ... * & ... & *... * \\
    & ...     & ...      & ...       &...  &   \\
0 ... 0 &0 ...0 & 0 ...0 & 0 ...0 &0 ...0& *... *
\end{pmatrix}
_{d \times (n+d)} .$$
The above row reductions gives the following formula 
$$u_i = \prod_{n_i \leq j \leq n+d } x_j ^{\beta_{i, j}} , ~ 1 \leq i \leq d, $$
and
$$u_i =u_i (t_2, t_3, ..., t_d)= \prod_{1\leq k\leq d} t_k^{\gamma_{i, k}}, ~ \gamma_{i, k}\in \mathbb{Z}$$
where  $1=n_1 < n_2 <... < n_d \leq n+d$.

Direct calculations show that  
$$d\log u_i  = \sum_{n_i\leq j \leq n+d} \beta_{i, j}  \frac{d x_j}{x_j} = \sum_{1 \leq k\leq d} \gamma_{i, k} \frac{dt_k}{t_k}, ~ 1 \leq i \leq d.$$
Since the upper diagonal $d\times d$-matrix $[\alpha_{i, n_i}]_{i =1, ...,d,}$ is invertible, for any $1\leq i \leq d$, we can solve 
$\frac{dx_{n_i}}{x_{n_i}}$ in terms of linear combination of $\frac{dx_j}{x_j}$ for  $j\geq n_i+1$ and $\frac{dt_k}{t_k}$ for $k=1, ... d$. 
If we define the ordered index set $J$ by
$$J= \{ 1\leq j \leq n+d ~|~ j\neq n_i, i =1, ..., d \}.$$
Obviously, $\# J = n$.
Then there exists a fixed constant $d\times n$ -matrix $[\mu_{i, j}]_{1\leq i \leq d, j \in J}$ with
$$ \mu_{i, j} = 0, ~ 1 \leq j\leq n_i $$
such that for each $t\in B$ and $1\leq i\leq d$
\begin{equation}\label{sec9abc}
\left. \frac{dx_{n_i}}{x_{n_i}} \right|_{\tilde{\mathcal{X}}_t \cap \mathcal{U}'} = \left. \sum_{j\in J} \mu_{i,j} \frac{dx_j}{x_j} \right|_{\tilde{\mathcal{X}}_t\cap \mathcal{U}'} .
\end{equation}
%
%
%and
%
%$$ \gamma_{i, j} = 0, ~ 2 \leq j\leq n_{i-1} +1$$
%
%because the  $(d-1)\times (d-1)$-matrix $[\alpha_{i, n_{i-1}+1}]_{i =2, ...,d,}$ is upper diagonal.
%
In particular, there exists $C_1>0$ such that  for each $t\in B$ and $1\leq i \leq d$, we have 
\begin{eqnarray} \label{sec9bca}
\left| \left.dx_{n_i} \right|_{\tilde{\mathcal{X}}_t \cap \mathcal{U}'} \right| &=& \left| \left. \sum_{j\in J,~j\geq n_i+1} \gamma_{i,j} \frac{x_{n_i}}{x_j} dx_j\right|_{\tilde{\mathcal{X}}_t\cap \mathcal{U}'} \right| \nonumber\\
&\leq&  \sum_{j\in J,~ j\geq n_i+1} |\gamma_{i,j}|  \left| \left.\frac{x_{n_i}}{x_j} dx_j\right|_{\tilde{\mathcal{X}}_t\cap \mathcal{U}'} \right|\\
&\leq & C_1 \sum_{j\in J, ~ j\geq n_i+1}  \left| \left. dx_j\right|_{\tilde{\mathcal{X}}_t\cap \mathcal{U}'} \right| \nonumber
\end{eqnarray}
because  $|x_{n_i}| \leq |x_j|$ for $j>n_i$.
Then immediately, by definition of $\Theta_2$, we have 
$$\Theta_2  |_{\tilde{\mathcal{X}}_t \cap \mathcal{U}'} = (\sqrt{-1})^n H  \left( \prod_{j\in J} d x_j \wedge d\bar x_j \right).$$
for a continuous algebraic function $H$ in $x_j$, $\bar x_j$,  $x_{n_i}  x_k^{-1}$, $\overline{ x_{n_i}  x_k^{-1}}$ for   $j=1, ..., n+d$, $i=2, ..., d$ and $n_i+1\leq k \in J$. Since $$\Theta_2|_{\tilde{\mathcal{X}}_t \cap \mathcal{U}\cap \{x_1=0\}}=0,$$
$G$ must have a factor as a bounded algebraic function in $x_1$,and $\bar x_1$. Therefore  there exist $0<\varepsilon<1$ and $C_2>0$ by (\ref{sec9bca}) such that for all $t\in B$, 
$$\Theta_2  |_{\tilde{\mathcal{X}}_t \cap \mathcal{U}'}  \leq (\sqrt{-1})^n C_2 \left. \left( |x_1|^{2\epsilon}  \prod_{j\in J} d x_j \wedge d\bar x_j \right)\right|_{\tilde{\mathcal{X}}_t \cap \mathcal{U}'}.$$
This implies that there exists $C_3>0$ such that for any $t\in B$, 
\begin{eqnarray*}
&&- \int_{\tilde{\mathcal{X}}_t \cap \mathcal{U}'}  \log |x_1|^2 ~\Theta_2 \\
&\leq& - (\sqrt{-1})^n C_2\int_{\tilde{\mathcal{X}}_t \cap \mathcal{U}'}  \left( \log |x_1|^2 \right) |x_1|^{2\epsilon} \prod_{j\in J} d x_j \wedge d\bar x_j  \\
&\leq& - (\sqrt{-1})^n C_2\int_{ |x_j|\leq 1,~j\in J}  \left( \log |x_1|^2 \right) |x_1|^{2\epsilon} \prod_{j\in J} d x_j \wedge d\bar x_j \\
&\leq& C_3.
\end{eqnarray*}
Also by definition of $\Theta_1$, there exists  $C_4>0$ such that for any $t\in B$,
$$\Theta_1 \leq C_4 (\sqrt{-1})^n \sum_{2\leq i_2< ...< i_n\leq n+d}     dx_1 \wedge d\bar x_1 \wedge dx_{i_2} \wedge d\bar x_{i_2} \wedge ... \wedge dx_{i_n}\wedge d\bar x_{i_n}$$
and so there exists $C_5>0$ such that for any $t\in B$, 
\begin{eqnarray*}
&&- \int_{\tilde{\mathcal{X}}_t \cap \mathcal{U}'}  \log |x_1|^2 ~\Theta_1 \\
&\leq&  - (\sqrt{-1})^n C_4\int_{\tilde{\mathcal{X}}_t \cap \mathcal{U}}  \log |x_1|^2  \sum_{2\leq i_2< ...< i_n\leq n+d} dx_1 \wedge d\bar x_1 \wedge dx_{i_2} \wedge d\bar x_{i_2} \wedge ... \wedge dx_{i_n}\wedge d\bar x_{i_n}\\
&\leq& - (\sqrt{-1})^n C_4 \sum_{2\leq i_2< ...< i_n\leq n+d}  \int_{ 0\leq |x_{i_1}|, ..., |x_{i_n}| \leq 1}  \log |x_1|^2   dx_1 \wedge d\bar x_1 \wedge dx_{i_2} \wedge d\bar x_{i_2} \wedge ... \wedge dx_{i_n}\wedge d\bar x_{i_n}\\
&\leq& C_5
\end{eqnarray*}
Immediately we have
$$- \int_{\tilde{\mathcal{X}}_t \cap \mathcal{U}'}  \log |x_1|^2 \chi^n \leq C_3+ C_5. $$
The lemma is then proved by replace $\mathcal{U}'$ by any subdomains of $\mathcal{U}$ as in (\ref{subdo}).

\end{proof}

The following lemma is proved in \cite{MA} in a more general setting for Deligne pairing. We provide a proof using the same argument in the proof of Lemma \ref{divisorinte}. 

\begin{lemma} \label{wpref2}  There exists $\psi_B \in \textnormal{PSH}(B)\cap L^\infty (B )$ such that 
\begin{equation}
\int_{\tilde{\mathcal{X}}_t} \chi^{n+1}= \ddbar \psi_B
\end{equation}
 on $B$.

\end{lemma}

\begin{proof} Let be $\sigma$ be a smooth holomorphic section of $[\chi]$ such that the support of $\sigma$ does not contain any fibre of $\tilde{\mathcal{X}} \rightarrow B$. Then $\chi = \ddbar \log |\sigma|^2_{h_\chi}$ for some smooth hermitian metric for $[\chi]$. After locally toric blow-ups, we can assume that the zeros of the pullback of $\sigma|^2_{h_\chi}$ in $\tilde{\mathcal{X}}$ in (\ref{ssr4}) and $\{x_j=0\}$ in (\ref{coordinates2}) have simple normal crossings. Then

\begin{eqnarray*}
\int_{\tilde{\mathcal{X}}_t} \chi^{n+1} &=&\int_{\tilde{\mathcal{X}}_t\setminus \{\sigma=0\}} \chi^{n+1} \\
&=& \int_{\tilde{\mathcal{X}}_t\setminus \{\sigma=0\}} \ddbar \log |\sigma|^2_{h_\chi} \wedge \chi^n\\
&=&\ddbar \left( \int_{\tilde{\mathcal{X}}_t\setminus \{\sigma=0\}} \log|\sigma|^2_{h_\chi} \chi^n\right).
\end{eqnarray*}
The same calculations in the proof of Lemma \ref{divisorinte} shows that there exists $C>0$ such that for all $t\in B$, 
$$-C \leq \psi_B= \int_{\tilde{\mathcal{X}}_t\setminus \{\sigma=0\}} \log|\sigma|^2_{h_\chi} \chi^n\leq C. $$
Since $\psi_B$ is bounded on $B$, $\ddbar \psi_B = \int_{\tilde{\mathcal{X}}_t} \chi^{n+1}$ and the lemma is proved.

\end{proof}

We remark that $\Psi_B$ is in fact H\"older continuous (c.f. \cite{MA}). 
We can now begin estimates of the local potentials of $\omega_{WP}$. 
We recall the following well-known functionals 
$$E_t: \textnormal{PSH}(\tilde{\mathcal{X}}_t, \chi_t) \rightarrow [-\infty, \infty) \cap L^\infty(\tilde{\mathcal{X}}_t), ~~G_t: \textnormal{PSH}(\tilde{\mathcal{X}}_t, \chi_t)  \cap L^\infty(\tilde{\mathcal{X}}_t)  \rightarrow [-\infty, \infty)
$$ by
\begin{eqnarray}
E_t(\phi_t) &=& \frac{1}{(n+1)V} \sum_{j=0}^n \int_{\tilde{\mathcal{X}}_t} \phi_t (\chi_t + \ddbar \phi_t)^j \wedge \chi_t ^{n-j}, \\
G_t(\phi_t) &=& E_t(\phi_t) - \log \left(  \frac{1}{V} \int_{\tilde{\mathcal{X}}_t} e^{\phi_t }\Omega_t \right), 
\end{eqnarray} 
where $t\in B^\circ$, $\phi_t \in \textnormal{PSH}(\tilde{\mathcal{X}}_t, \chi_t)$ and $V = [ K_{\tilde{\mathcal{X}}_t}]^n$ is independent of $t\in B^\circ$. 
$G_t$ is the analogue of the Ding-functional introduced in \cite{D}. Recall that  $\varphi_t= \varphi|_{\tilde{\mathcal{X}}_t}$ is the unique solution of the complex Monge-Amp\`ere equaiton
\begin{equation} \label{kesec9}
(\chi_t + \ddbar \varphi_t)^n = e^{\varphi_t} \Omega_t
\end{equation}
induced by the K\"ahler-Einstein equation on $\tilde{\mathcal{X}}_t$ for all $t\in B^\circ$. We define the function $E(t)$ on $B^\circ$ by 
\begin{equation}
E(t) = E_t(\varphi_t),
\end{equation}
where $t\in U^\circ$ and $\varphi_t$ is the unique solution of equation (\ref{kesec9}).  %
Immediately, we have the following lemma by Lemma \ref{wpref}. 
\begin{lemma} \label{wppot} On $B^\circ$, 
\begin{equation}
E(t) = G_t(\varphi_t),  %
\end{equation}
\begin{equation}
\omega_{WP} =  \ddbar \left(\psi_U +   E(t) \right),
\end{equation}
where  $\psi_U$ is defined in Lemma \ref{wpref2}.
\end{lemma}

Our goal for the rest of the section is then to bound $E(t)$ uniformly. The following theorem is the key result of this section. 

\begin{theorem} \label{main91} There exists $C>0$ such that 
\begin{equation}
\inf_{t\in B^\circ} \sup_{\phi_t \in \textnormal{PSH}(\tilde{\mathcal{X}}_t, \chi_t)\cap L^\infty(\tilde{\mathcal{X}}_t)} G_t (\phi_t) \geq - C. 
\end{equation}

\end{theorem}

Theorem \ref{main91} will imply that the potential $\psi_U + E(t)$ is uniformly bounded on $B^\circ$ and so it can be extended to $B$.
\begin{corollary} \label{bndet} There exists $C>0$ such that 
\begin{equation}
\sup_{t\in B^\circ} | E(t) | \leq C. 
\end{equation}

\end{corollary}

\begin{proof}  We first observe that by definition of $\varphi_t$
$$\int_{\cX_t} e^{\varphi_t} \Omega_t = \int_{\cX_t} (\chi_t + \ddbar \varphi_t)^n = V. $$
The K\"ahler-Einstein potential $\varphi_t$ maximizes the functional $G_t$ for each $t\in B^\circ$ by the same argument in \cite{DT} and it is generalized to the singular case in \cite{BBGZ}. By Theorem \ref{main91}, there exists $C_1>0$ such that
$$E(t) = E_t(\varphi_t) - \log \left(  \frac{1}{V} \int_{\tilde{\mathcal{X}}_t} e^{\varphi_t }\Omega_t \right)= G_t(\varphi_t) = \sup_{\phi_t \in \textnormal{PSH}(\tilde{\mathcal{X}}_t, \chi_t)\cap C^\infty(\tilde{\mathcal{X}}_t)} G_t(\phi_t) \geq - C_1 $$
for all $t\in B^\circ$.
On the other hand, by Lemma \ref{5upb}, there exists $C_2>0$ such that for each $t\in B^\circ$, 
$$\varphi_t  \leq C_2.$$
Therefore, 
\begin{eqnarray*}
E(t) &=&  \frac{1}{(n+1)V} \sum_{j=0}^n \int_{\tilde{\mathcal{X}}_t} \varphi_t (\chi_t + \ddbar \varphi_t)^j \wedge \chi_t ^{n-j}\\
&\leq& \frac{C_2}{(n+1)V} \sum_{j=0}^n \int_{\tilde{\mathcal{X}}_t}  (\chi_t + \ddbar \varphi_t)^j \wedge \chi_t ^{n-j}\\
&=& C_2.
\end{eqnarray*}
This completes the proof of the corollary.
\end{proof}

The following corollary naturally holds by letting $\varphi_{WP} = \psi_B + E(t)$ from Corollary \ref{wppot} and Corollary \ref{bndet}.
\begin{corollary} \label{wpex} There exists $\varphi_{WP} \in \textnormal{PSH}(B)\cap L^\infty(B)$ such that 
\begin{equation}
\omega_{WP}=\ddbar \varphi_{WP}
\end{equation}
on $B$. In particular, the Weil-Petersson metric $\omega_{WP}$ uniquely extends to a closed positive $(1,1)$-current on $B$ with bounded potentials. 

\end{corollary}

We will now prove Theorem \ref{main91}.  Let $\mathcal{E}$ is the exceptional locus of $\Psi\cdot\Phi$. Then $\Psi\circ\Phi(\mathcal{E})$ is a Zariski closed set of $\mathcal{Z}$ and we let $\mathcal{I}$ be the ideal sheaf of $\Psi\circ\Phi(\mathcal{E})$ and $\mathcal{I}$. Since $K_{\mathcal{Z}/B}$ is ample on $\mathcal{Z}$ and $\mathcal{I}$ is a coherent sheaf, $mK_{\mathcal{Z}/B} \otimes \mathcal{I}$ is globally generated for some sufficiently large $m$. Therefore there exist  holomorphic sections  $\sigma_1, ..., \sigma_M$ of $mK_{\mathcal{Z}/B}$ satisfying
$$\{\sigma_1=...=\sigma_M =0\} = \Psi\circ\Phi(\mathcal{E}). $$  Let $h$ be a smooth hermitian metric on $K_{\mathcal{Z}/B}$ satisfying $-\ddbar\log h = \chi$. We can assume that $\log \left( \sum_{j=1}^M |\sigma_j|^2_h \right) \leq -1$ on $\mathcal{Z}$ by scaling $h$. 
We now define 
\begin{equation} \label{psidef}
\psi= - \left( - \frac{1}{m} \log \left( \sum_{j=1}^M |\sigma_j|^2_h \right) \right)^\alpha
\end{equation}
for some $\alpha\in (0, 1)$ to be determined later. For simplicity, we also identify $(\Psi\circ\Phi)^*\psi$ with $\psi$.
\begin{lemma} $\psi \in \textnormal{PSH}(\mathcal{Z}, \chi)$, i.e., 
$$\chi + \ddbar \psi \geq 0.$$

\end{lemma}

\begin{proof} First, we observe that $f=\log \left( \sum_{j=1}^M |\sigma_j|^2_h \right)\in \textnormal{PSH}(\mathcal{Z}, \chi).$ Then 
\begin{eqnarray*}
\chi+ \ddbar \left(- (-f)^\alpha\right) &=& \chi+ \alpha  (-f)^{-(1-\alpha)} \ddbar f + \alpha(1-\alpha) \sqrt{-1} (-f)^{-(2-\alpha)} \partial f \wedge \dbar f \\
&\geq& 0
\end{eqnarray*}
because $\alpha\in(0,1)$ and $f\leq -1$.
\end{proof}

In local coordinates on an open domain $\mathcal{U} \subset \tilde{\mathcal{X}}_t$, there must be at least $n$ exceptional divisors defined by $\{x_{i_k}=0\}$, where $1\leq i_1< i_2<... < i_m\leq n+d$ and $k=1, ..., m$. After rearrangement, %
 we can consider  
\begin{equation}\label{locco9}
t_i = \left( \prod_{k=1}^{d} y_k^{\alpha_{i,k}}\right) \left( \prod_{j=1}^{n} x_j^{\beta_{i, j}} \right),~ i=1, ..., d,
\end{equation}
so that  $d \times d$-matrix $[\alpha_{i, k}]$ has rank $d$ and the exceptional locus of $\Psi\circ \Phi$ contains zeros of $x_1$, ..., $x_n$. In particular,  $dx_1\wedge d\bar x_1\wedge ... \wedge dx_n \wedge d\bar x _n$ is a nonzero volume on proper transformation of each fibre. We can assume that $\cup_{i=1^n} \{x_i=0\} \cup_{k=1}^{m-n} \{y_k=0\}$ is the exceptional divisor of $\Psi\circ\Phi$ in $\mathcal{U}$ for some $n\leq m \leq n+d$.

\begin{lemma} \label{omegaest} In local coordinates on an open domain $\mathcal{U} \subset \tilde{\mathcal{X}}_t$ as in (\ref{locco9}). There exists $C>0$ such that  for any $t\in B$,
\begin{equation}
(\Psi\circ\Phi)^*\Omega|_{ \tilde{\mathcal{X}}_t\cap \mathcal{U} } \leq \frac{C (\sqrt{-1})^ndx_1\wedge d\bar x_1\wedge ... \wedge dx_n \wedge d\bar x _n}{ |x_1|^2|x_2|^2... |x_n|^2}.
\end{equation}

\end{lemma}

\begin{proof}   Straightforward calculations show that $$\frac{dt_i}{t_i} = \sum_{k=1}^d  \frac{\alpha_{i, k} dy_k}{y_k} + \sum_{j=1}^n \frac{\beta_{i, j} dx_j}{x_j}. $$
Since $[\alpha_{i, k}]$ is an invertible $d \times d$-matrix, one can solve $\frac{dy_k}{y_k}$ in terms of $\frac{dx_j}{x_j}$ and $\frac{dt_i}{t_i}$.
Then we have 
\begin{eqnarray*}
&&(\Psi\circ\Phi)^*\Omega|_{\tilde{\mathcal{X}}\cap \mathcal{U} }\\
&=&(\sqrt{-1})^n \sum_{k+l=n} F_{i_1...i_k, j_1...j_l} dy_{i_1}\wedge d\bar y_{i_1} \wedge ... \wedge dy_{i_k}\wedge d \bar y_{i_k} \wedge dx_{j_1}\wedge d\bar x_{j_1} \wedge ... \wedge dx_{j_l}\wedge d \bar x_{j_l}\\
&=&   \Theta + \Theta', 
\end{eqnarray*}
where
$$ \Theta =  (\sqrt{-1})^nF(y_1, ..., y_d, x_1, ...., x_n) dx_1\wedge d\bar x_1 \wedge ... \wedge dx_n \wedge d \bar x_n$$
%}{|x_1|^2 |x_2|^2 ... |x_n|^2} $$
%
and $\Theta'$ contains one of $dt_1, d\bar t_1, ..., dt_d, d\bar t_d.$ In particular, 
$$\Theta'|_{\tilde{\mathcal{X}}_t\cap \mathcal{U} } = 0. $$
%
%Without loss of generality, we can assume that none of $\{y_1=0\}$, ..., $\{y_q=0\}$ is an exceptional divisor and $\{y_{q+1}=0\}$, ..., $\{y_d=0\}$ are exceptional divisors. 
There exist $C_1, C_2>0$ such that on $\tilde{\mathcal{X}}\cap U$, 
\begin{eqnarray*}
&& \frac{ (\sqrt{-1})^{n+d} dt_1\wedge d\bar t_1\wedge...\wedge d t_d\wedge d\bar t_d\wedge (\Psi\circ\Phi)^*\Omega}{|t_1|^2... |t_d|^2} \\
&=& \frac{ (\sqrt{-1})^{n+d}  dt_1\wedge d\bar t_1\wedge...\wedge d t_d\wedge d\bar t_d\wedge \Theta} {|t_1|^2... |t_d|^2}\\
&\leq&  \frac{(\sqrt{-1})^{n+d} C_1 dy_1\wedge d\bar y_1\wedge...\wedge d y_d \wedge d 
\bar y_d\wedge dx_1\wedge d\bar x_1 \wedge ... \wedge dx_n \wedge d \bar x_n}{|y_1|^2...|y_d|^2 |x_1|^2... 
|x_n|^2 }\\
&\leq&  \frac{ (\sqrt{-1})^{n+d}  C_2 dt_1\wedge d\bar t_1\wedge...\wedge d t_d \wedge d \bar t_d \wedge dx_1\wedge d\bar x_1 \wedge ... \wedge dx_n \wedge d \bar x_n}{|t_1|^2...|t_d|^2 |x_1|^2... 
|x_n|^2 }, 
\end{eqnarray*}
where the first inequality follows from (\ref{adjun4forsec9}) in the proof   of Lemma \ref{adjun4}.
Therefore there exists $C_3>0$ such that 
$$F(y_1, ..., y_d, x_1, ..., x_n) \leq \frac{C_3} { |x_1|^2... 
|x_n|^2 },
$$
or equivalently,
$$\Theta \leq \frac{(\sqrt{-1})^n C_3 dx_1\wedge d\bar x_1\wedge ... \wedge dx_n \wedge d\bar x _n}{ |x_1|^2|x_2|^2... |x_n|^2}.$$
The lemma then immediately follows.
\end{proof}

\begin{lemma} \label{sec97} We keep the same notions as in (\ref{locco9}) and assume that the exceptional locus of $\Psi\circ\Phi$ lies in the divisor locally defined by $\cup_{j=1}^n \{ x_j=0\} \cup_{k=1}^{m-n}\{ y_k =0\}$ for some $n \leq m \leq n+d$.  Then there exists $C>0$ such that in $\mathcal{U}$, we have 
\begin{equation}\label{sec922}
\psi \geq  C \log |y_1... y_{m-n} x_1... x_n|^2  -C
\end{equation}
and
\begin{equation} \label{sec923}
\psi   \leq - 2n\log \left(- \log |y_1... y_{m-n} x_1... x_n|^2 \right) +C . 
\end{equation}

\end{lemma}

\begin{proof} Estimate (\ref{sec922}) follows from the definition of $\psi$ where the set $\{\sigma_1=... =\sigma_M=0\}$ coincides with the exceptional locus of $\Psi\circ\Phi$. Estimate (\ref{sec923}) follows by the same observation and the simple fact that $x^\alpha \geq 2n \log x + C_\alpha$ for $x>0$ and $\alpha\in (0,1)$.

\end{proof}

\begin{lemma} \label{sec990} There exists $C>0$ such that for all $t\in B$, 
\begin{equation}
\int_{\tilde{\mathcal{X}}_t} e^{\psi_t} \Omega_t \leq C,
\end{equation}
where $\psi_t = \psi|_{\tilde{\mathcal{X}}_t}$ and $\psi$ is defined in (\ref{psidef}). 

\end{lemma}

\begin{proof} We use the same notions as in (\ref{locco9}) and Lemma \ref{sec97}. The lemma then immediately follows from straightforward integral calculations, Lemma \ref{sec97} and  Lemma \ref{omegaest}.

\end{proof}

\begin{lemma} \label{sec9101} For any $\alpha \in (0,1)$, there exists $C_\alpha>0$ such that for all $t\in B$, 
\begin{equation}
- \int_{\tilde{\mathcal{X}}_t}   (-\psi_t )^{\frac{1}{\alpha}} \chi_t^n \leq C_\alpha,
\end{equation}
or equivalently, 
\begin{equation}
- \int_{\tilde{\mathcal{X}}_t}    \log \left(\sum_{j=1}^M |\sigma_j|^2_h   \right)\chi_t^n \leq C_\alpha.
\end{equation}

\end{lemma}

\begin{proof} The lemma follows from estimate (\ref{sec922}) and Lemma \ref{divisorinte}.

\end{proof}

In order to prove Theorem \ref{main91}, it suffices to show that there exists $C>0$ such that for all $t\in B^\circ$, 
\begin{equation}\label{gtest}
G_t(\psi_t) \geq - C.
\end{equation}
We will need the capacity argument in the pluripotential theory to achieve   (\ref{gtest}). We recall the definition of complex capacity associated to plurisubharmonic functions. 

\begin{definition} Let  $(X, \omega)$ be a compact K\"ahler manifold of complex dimension $n$. The Monge-Amp\`ere capacity for a Borel subset of $X$ associated to $\omega$ is defined by 
\begin{equation}
\textnormal{Cap}_\omega(K) = \sup\left\{ \int_K (\omega+ \ddbar u)^n ~|~ u \in \textnormal{PSH}(X, \omega)\textnormal{,} ~¡« -1\leq u \leq  0\right\}. 
\end{equation}

\end{definition}

The following lemma is proved in \cite{GZ1} (Proposition 2.6). We  include the proof since it is quite short and elementary.

 \begin{lemma} \label{cap2} Let  $(X, \omega)$ be a compact K\"ahler manifold of complex dimension $n$. For any $\lambda>1$ and $\phi \in \textnormal{PSH}(X, \omega)$ with $\phi<0$, we have
 \begin{equation}
 \textnormal{Cap}_\omega(\phi< -\lambda) \leq \lambda^{-1} \left( \int_X (-\phi)\omega^n + n \int_X \omega^n \right).
 \end{equation}
 
 \end{lemma}
 
\begin{proof} For any $u\in \textnormal{PSH}(X, \omega)$ with $-1< u<0$, we let $v=u+1$. Then $0<v<1$. 
For any $\lambda >1$, 
$$\int_{\{\phi < -\lambda\}} \omega_v^n \leq \int_X \left(-\frac{\phi}{\lambda} \right) \omega_v^n = \lambda^{-1} \int_X (-\phi)\omega_v^n,$$
where $\omega_v= \omega+\ddbar v$.  Straightforward calculations show that 
\begin{eqnarray*}
&&\int_X (-\phi) \omega_v^n \\
&=& \int_X (-\phi)\omega^n + \sum_{k=0}^{n-1} \int_X (-\phi)\ddbar v \wedge \omega^k \wedge\omega_v^{n-k+1}\\
&=& \int_X (-\phi)\omega^n +\sum_{k=0}^{n-1} \int_X  v \omega^{k+1} \wedge \omega_v^{n-k+1}-  \sum_{k=0}^{n-1} \int_X v (\omega+ \ddbar \phi) \wedge \omega^k \wedge \omega_v^{n-k+1}\\
&\leq& \int_X (-\phi)\omega^n+ \sum_{k=0}^{n-1} \int_X \omega^{k+1}\wedge\omega_v^{n-k+1}\\
&=& \int_X (-\phi)\omega^n+n \int_X \omega^n.
\end{eqnarray*}
The lemma is then proved by the definition of $\textnormal{Cap}_\omega(\phi< -\lambda)$.

\end{proof}

The following lemma is implicitly proved in \cite{GZ2} (Lemma 5.1). We include a short proof for completeness.

\begin{lemma} \label{cap1} Let  $(X, \omega)$ be a compact K\"ahler manifold of complex dimension $n$.  If $\phi\in \textnormal{PSH}(X, \omega)\cap L^\infty(X)$ with $\phi<-1$ and there exists $A>0$ such that for any  $\lambda >1$,  
\begin{equation}
\textnormal{Cap}_\omega (\phi < -\lambda) \leq A \lambda^{-(n+2)},
\end{equation}
then
\begin{equation}
\int_X (-\phi) (\omega+\ddbar\phi)^n \leq  A. 
\end{equation}

\end{lemma}

\begin{proof}  
We let $\psi=\max (\phi, -\lambda)$ and $u = \lambda^{-1} \psi$. Then $-1\leq u\leq 0$ and
\begin{eqnarray*}
\int_{\{\phi< - \lambda\}} (\omega+\ddbar\phi)^n&=&\int_X \omega^n - \int_{\{\phi \geq -\lambda\}} (\omega+\ddbar \phi)^n\\
&=& \int_X (\omega+\ddbar\psi)^n - \int_{\{ \phi \geq -\lambda\}} (\omega + \ddbar \psi)^n\\
&=& \int_{\{ \phi < -\lambda \}} (\omega + \ddbar \psi)^n \\
&\leq& \lambda^n \int_{\{ \phi <-\lambda \}} (\omega + \ddbar u)^n \\
& \leq & \lambda^n \textnormal{Cap}_\omega( \phi < -\lambda) \\
&\leq& A\lambda^{-2}. 
\end{eqnarray*}
%
%$$\int_{\{\phi < -\lambda\}} \lambda^{-n} (\omega+\ddbar\phi)^n\leq  \int_{\{ \phi<-\lambda \}} (\omega+\ddbar u)^n  \leq \textnormal{Cap}_\omega( \phi < -\lambda)\leq \frac{A}{\lambda^{n+2}}=A. $$
%
The co-area formula implies that 
\begin{eqnarray*}
\int_X (-\phi) (\omega+\ddbar \phi)^n = \int_1^\infty \left( \int_{\{\phi < -\lambda \}} (\omega+\ddbar \phi)^n  \right) d\lambda
\leq \int_1^\infty A\lambda^{-2} d\lambda =A.
\end{eqnarray*}

\end{proof}
 
Now we are ready to complete the proof of Theorem \ref{main91}.

\medskip

\noindent {\bf Proof of Theorem \ref{main91}.} 
We first observe that since $\psi$ is nonpositive. We let $\psi'= \psi-1$. There exists $\delta_n>0$ such that 
\begin{eqnarray*}
E_t(\psi_t) &=&E_t(\psi_t') +1 \\
&=& \frac{1}{(n+1)V} \sum_{j=0}^n \int_{\tilde{\mathcal{X}}_t} \psi'_t ~(\chi_t + \ddbar \psi'_t)^j \wedge \chi_t ^{n-j} +1\\
&\geq &\frac{\delta_n}{(n+1)V} \left( \int_{\tilde{\mathcal{X}}_t}\psi'_t ~ \chi_t^n + \int_{\tilde{\mathcal{X}}_t}\psi'_t~(\chi_t+ \ddbar \psi'_t)^n \right)+1.
\end{eqnarray*}
By Lemma \ref{sec9101}, for fixed sufficiently small $\alpha>0$, there exists $C_1>0$ such that for all $t\in B^\circ$, 
$$ -\int_{\tilde{\mathcal{X}}_t} -\psi'_t  ~\chi_t^n \geq -C_1.$$
By Lemma \ref{cap2}, there exists $C_2>0$ such that for all $t\in B^\circ$ and $\lambda>1$,
\begin{eqnarray*}
\textnormal{Cap}_{\chi_t} ( \psi'_t < - \lambda) &=&\textnormal{Cap}_{\chi_t} \left(  - \left( -\psi'_t \right)^{\frac{1}{\alpha}}< - \lambda^{\frac{1}{\alpha} }\right) \\
&\leq& \lambda^{-\frac{1}{\alpha}} \left( \int_X -(\psi'_t)^{\frac{1}{\alpha}} \chi_t^n + n \int_{\tilde{\mathcal{X}}_t} \chi_t^n\right)\\
&\leq& C_2 \lambda^{-\frac{1}{\alpha}}.
\end{eqnarray*}
We can apply Lemma \ref{cap1} by choosing $\alpha = \frac{1}{n+2}$ so there exists $C_3>0$ such that for all $t\in B^\circ$, 
$$\int_X \psi'_t ~(\chi_t + \ddbar \psi'_t)^n \leq C_3. $$
Immediately, there exists $C_4>0$ such that 
$$E_t(\psi_t) \geq - C_4.$$
By Lemma \ref{sec990}, there exists $C_5>0$ such that for all $t\in B^\circ$, 
\begin{eqnarray*}
G_t(\psi_t) &=& E_t(\psi_t) - \log \left( \frac{1}{V} \int_{\mathcal{X}_t} e^{\psi_t} \Omega_t \right)\\
&\geq& - C_5. 
\end{eqnarray*}
Equivalently,  for all $t\in B$ 
$$\sup_{\phi_t\in \textnormal{PSH}(\mathcal{X}_t, \chi_t)\cap L^\infty(\mathcal{X}_t)} G_t (\phi_t) \geq -C_5.$$
This proves Theorem \ref{main91}.   $\hfill\Box$

\bigskip

Now we will complete the proof of Theorem \ref{main3}.

\medskip

\noindent {\bf Proof of Theorem \ref{main3}.} Let $\pi: \mathcal{X} \rightarrow  S$ be a stable family of $n$-dimensional   canonical models over an $m$-dimensional projective normal variety $S$. We apply  semi-stable reduction and (locally toric) blow-ups as in (\ref{ssrsec9}) and (\ref{ssr4}).
\begin{equation} 
\begin{diagram}
\node{\tilde{\mathcal{X}}}  \arrow{e,t}{\Phi} \arrow{se,r}{\tilde{\pi}} \node{\mathcal{X}'} \arrow{s,l}{ \pi' }  \arrow{e,t}{\Psi}   \node{\cZ:=\mathcal{X}\times_S  S'}  \arrow{sw,t}{\pi_{\mathcal{Z}}} \arrow{e,t}{f'}   \node{\mathcal{X} } \arrow{sw,r}{\pi} \\
\node{}      \node{S'} \arrow{e,t}{f}  \node{S}
\end{diagram}
\end{equation}
Let $\omega_{\WP}$ be the smooth Weil-Petersson metric on $S^\circ$ and let $\tilde{\omega}_{\WP}= f^*\omega_{\WP}$ on $(S')^\circ= f^{-1}(S^\circ)$. By  Corollary \ref{wpex}, $\tilde{\omega}_{WP}$ extends to a nonnegative closed $(1,1)$-current on $S'$ with bounded local potentials. If we let 
$$\mathcal{L}= \langle \underbrace{K_{\mathcal{X}/S}, \cdots, K_{\mathcal{X}/S} }_{n+1}\rangle$$
be the CM line bundle induced by the $(n+1)$-fold Deligne pairing of of $K_{\mathcal{X}/S}$ over $S$, then $\omega_{\WP}$ is the curvature of $\mathcal{L}$ on $S^\circ$ and since $\tilde{\omega}_{WP}$ has bounded local potentials on $S'$, $\tilde{\omega}_{\WP}$ is the curvature of $f^*\mathcal{L}$ globally on $S'$ and 
$$\tilde{\omega}_{\WP}\in [f^*\mathcal{L}].$$ Therefore $\omega_{\WP}$ must also extend to a  nonnegative closed $(1,1)$-current on $S$ with bounded local potentials with $\omega_{\WP}\in   [\mathcal{L}]$. In particular, 
$$\int_{\mathcal{S}} (\omega_{WP})^m =\int_{\mathcal{S}^\circ} (\omega_{WP})^m = [\mathcal{L}]^m \in \mathbb{Q}.$$. $\hfill\Box$

\medskip

\noindent {\bf Proof of  Corollary \ref{main4}.} We will first explain how the Weil-Petersson metric is defined on the moduli space. We let $\overline{\cM}_\KSBA$ be the KSBA compactification of the moduli space $\cM_\KSBA$ for a smooth $n$-dimensional canonical model $X$. Here we always assume the generic point of $\overline{\cM}_\KSBA$ corresponds to a smooth canonically polarized manifold. To  construct the Weil-Petersson current on $\overline{\cM}_\KSBA$, we will have to replace 
$\overline{\cM}_\KSBA$ by the base of a stable family as in Theorem \ref{main3}. Since the automorphism group of semi-log canonical models is finite, we can apply results in \cite{K3} (c.f. Proposition 2.7) and \cite{HX} that there exists  a  finite morphism 
\begin{equation}
\phi: \mathcal{S} \to \overline{\cM}_\KSBA
\end{equation}
 from a normal projective variety $\mathcal{S}$, together with a stable family of canonical models
 $$\pi: \cX\to \cS$$  such that for any $s\in \mathcal{S}$ the moduli point $\cX_s=\pi^{-1}(s)$ is exactly $\phi(s)\in   
{\cM}_\KSBA$.  

Let $\omega_{WP, \cS}$ be the Weil-Petersson current on $\mathcal{S}$. $\omega_{WP, \cS}$ is a smooth K\"ahler metric on a Zariski open set of $\mathcal{S}$ and by Theorem \ref{main3}, $\omega_{WP, \cS}$ has bounded local potential and $[\omega_{WP, \cS}] \in c_1(\mathcal{L}_\cS)$, where $\mathcal{L}_\cS$ is the CM line bundle induced by the Deligne pairing of $K_{\cX /\mathcal{S}}$. Therefore, there exists a Zariski open set $\left({\cM}_\KSBA\right)^\circ$ of ${\cM}_\KSBA$ such that 
$\left({\cM}_\KSBA\right)^\circ$ is smooth and the Weil-Petersson current $\omega_{WP, \cS}$ is a smooth K\"ahler metric on $\phi^{-1} \left( \left({\cM}_\KSBA\right)^\circ \right)$. Immediately $\omega_{WP, \cS}$ descends to a smooth K\"ahler metric on $\left({\cM}_\KSBA\right)^\circ$ by pushforward or averaging. We define such a K\"ahler metric as the Weil-Petersson metric on $\left({\cM}_\KSBA\right)^\circ$ and denote it by $\omega_{WP}$. We also notice that as in \cite{PX}, the CM line bundle $\mathcal{L}_\cS$ on $\cS$ is the pullback of the CM line bundle $\mathcal{L}$ on $\overline{\cM}_\KSBA$.

The construction of $\omega_{WP}$ on $\left({\cM}_\KSBA\right)^\circ$ does not depend on the choice of $\phi: \mathcal{S} \to \overline{\cM}_\KSBA$.  If there exists another finite map $\phi': \cS' \rightarrow \overline{\cM}_\KSBA$ together with a stable family $\pi': \cX' \rightarrow \cS$ such that for any $s\in \mathcal{S}'$ the moduli point $\cX'_s=\pi'^{-1}(s)$ is exactly $\phi'(s)\in   
\overline{\cM}_\KSBA$.  One can further find a finite map $\tilde{\phi}: \tilde{\cS} \rightarrow \overline{\cM}_\KSBA$ together with a stable family $\tilde{\pi}: \tilde{\cX} \rightarrow \tilde{\cS}$  
%for any $s\in \tilde{\mathcal{S}}$ the moduli point $\tilde{\cX}_s=\tilde{\pi}^{-1}(s)$ is exactly $\tilde{\phi}(s)\in   {\cM}_\KSBA$, 
satisfying the commutative diagram
\begin{equation}
\xymatrix{
&\psi^\ast\cX  \ar@{>}[dr]_{}& \tilde{ \cX} \ar@{>}[d]^{\tilde \pi} \ar@{>}[l]_{ \cong } \ar@{>}[r]^{\cong\ \ \  } &\left(\psi'\right)^\ast\cX'  \ar@{>}[dl]^{}&\\
& & \tilde{\cS}\ar@{>}[dr]^{\psi'} \ar@{>}[dd]^{} \ar@{>}[dl]_{\psi } &  &\\
\cX \ar@{>}[r]_{\pi}& \cS  \ar@{>}[dr]_{\phi} & & \cS'\ar@{>}[dl]^{\phi'} &  \cX' \ar@{>}[l]^{\pi'}\\
& &  \overline{\cM}_\KSBA & &
            }
\end{equation}
by the fact that the automorphism group of a semi-log canonical model is finite. {In particular, we have}
$$\psi^\ast\cX\cong\ti \cX\cong\psi'^\ast \cX'.$$

The Weil-Petersson metric $\omega_{WP}$ can be constructed from $\cS$, $\cS'$ and $\tilde{\cS}$, and all the constructions coincide  on a Zariski open set of $\overline{\cM}_\KSBA$. The above observation immediately shows that
$$\phi^* \omega_{WP} = \omega_{WP, \cS}$$
on a Zariski open set of ${\cM}_\KSBA$.

After normalization and resolution of singularities we can assume $\cS$ is smooth and we can  extend $\omega_{WP}$ globally on $\overline{\cM}_\KSBA$ by pushing forward $\omega_{WP,\cS}$ via $\phi$. We would like to show that the extended $\omega_{WP}$ has bounded local potentials.  Since $\mathcal{L}_\cS=\phi^*\mathcal{L}$, we can pick a closed $(1,1)$-current $\eta \in c_1(\mathcal{L})$ on $\overline{\cM}_\KSBA$, where $\eta$ is the restriction of some Fubini-Study metric to $\overline{\cM}_\KSBA$ embedded in an ambient projective space. Then $\phi^*\eta \in c_1(\mathcal{L})$ and there exists $\varphi \in \textnormal{PSH}(\cS, \phi^*\eta)\cap L^\infty(\cS)$ such that
$$\phi^*\omega_{WP}= \phi^*\eta+ \ddbar \varphi.$$
Now for any $p\in \overline{\cM}_\KSBA$ and a sufficiently small affine neighborhood $U$ of $p$ in $\overline{\cM}_\KSBA$, there exists $\psi \in \textnormal{PSH}(\phi^{-1}(U))\cap L^\infty(\phi^{-1}(U))$ such that
$$\ddbar \psi = \omega_{WP}.$$
Let $\psi$ be the upper envelope of the push-forward of $\phi$ by $\phi$ on $U$. Then we have
$$\phi^* \omega_{WP} = \ddbar \left(\phi^* \psi\right).$$

Now we can define the Weil-Petersson volume by 
$$
\int_{\overline{\cM}_\KSBA}\left(\omega_\WP \right)^d=\int_{\left({\cM}_\KSBA\right)^\circ} \left(\omega_\WP\right)^d,$$
where $d=\dim {\cM}_\KSBA$. Since $\phi^* \omega_{WP}=\omega_{WP, \cS}$ has bounded local potentials, 
\begin{eqnarray*}
\int_{\left({\cM}_\KSBA\right)^\circ} \left(\omega_\WP\right)^d
&=&\frac{1}{\deg \phi}\int_{ \phi^{-1} \left( \left({\cM}_\KSBA\right)^\circ \right) } \left(\omega_{WP, \cS}\right)^d=\frac{1}{\deg \phi}\int_{ \cS } \left(\omega_{WP, \cS}\right)^d \\
&=& \frac{1}{\deg \phi} \left(c_1(\mathcal{L}_\cS) \right)^d= \left( c_1(\mathcal{L})\right)^d \in \mathbb{Q}^+.
\end{eqnarray*}

This completes the proof of Corollary \ref{main4}.
$\hfill\Box$

\end{document}